\documentclass[reqno]{amsart}

\usepackage[utf8]{inputenc}
\usepackage{amssymb}
\usepackage{amsthm}
\usepackage{amsmath}
\usepackage{tikz}

\usepackage[toc,page]{appendix}

\usepackage{enumitem}

\usepackage{tikz-cd}

\usepackage{tikz-cd}

\usepackage{tikz}
\usetikzlibrary{datavisualization}
\usetikzlibrary{datavisualization.formats.functions}

\usetikzlibrary{arrows}

\usepackage{import}

\usepackage{mathtools}
\usepackage{mathrsfs}
\usepackage{commath}

\usepackage{framed}

\usetikzlibrary{calc}

\makeatletter
\pgfarrowsdeclare{X}{X}
{
  \pgfutil@tempdima=0.3pt%
  \advance\pgfutil@tempdima by.25\pgflinewidth%
  \pgfutil@tempdimb=5.5\pgfutil@tempdima\advance\pgfutil@tempdimb by.5\pgflinewidth%
  \pgfarrowsleftextend{+-\pgfutil@tempdimb}
  \pgfarrowsrightextend{0pt}
}
{
  \pgfutil@tempdima=0.3pt%
  \advance\pgfutil@tempdima by.25\pgflinewidth%
  \pgfsetdash{}{+0pt}
  \pgfsetroundcap
  \pgfsetmiterjoin
  \pgfpathmoveto{\pgfqpoint{-5.5\pgfutil@tempdima}{-6\pgfutil@tempdima}}
  \pgfpathlineto{\pgfqpoint{5.5\pgfutil@tempdima}{6\pgfutil@tempdima}}
  \pgfpathmoveto{\pgfqpoint{-5.5\pgfutil@tempdima}{6\pgfutil@tempdima}}
  \pgfpathlineto{\pgfqpoint{5.5\pgfutil@tempdima}{-6\pgfutil@tempdima}}
  \pgfusepathqstroke 
}
\makeatother

\newtheorem{theorem}{Theorem}[subsection]
\newtheorem{lemma}[theorem]{Lemma}

\theoremstyle{definition}
\newtheorem{definition}[theorem]{Definition}
\newtheorem*{acknowledgements}{Acknowledgements}
\newtheorem{example}[theorem]{Example}
\theoremstyle{remark}
\newtheorem{remark}[theorem]{Remark}

\usepackage{relsize}
\usepackage{exscale}

\usepackage{mathtools}
\usepackage{mathrsfs}

\usepackage{framed}

\usepackage{datetime}

\usepackage{hyperref} 

\usepackage{pifont}

\usepackage[pagewise]{lineno}

\newcommand{\y}{\ding{51}}
\newcommand{\x}{\ding{55}}

\renewcommand{\epsilon}{\varepsilon}

\DeclareMathOperator{\Orb}{Orb}

\title{Preservation of shadowing in discrete dynamical systems}
\author{Chris Good} 
\author{Joel Mitchell}
\author{Joe Thomas}
\date{July 2019}

\begin{document}

\hypersetup{pageanchor=false} 


\begin{abstract} 
We look at the preservation of various notions of shadowing in discrete dynamical systems under inverse limits, products, factor maps and the induced maps for symmetric products and hyperspaces. The shadowing properties we consider are the following: shadowing, h-shadowing, eventual shadowing, orbital shadowing, strong orbital shadowing, the first and second weak shadowing properties, limit shadowing, s-limit shadowing, orbital limit shadowing and inverse shadowing.
\end{abstract}

\maketitle

\hypersetup{pageanchor=true} 

Let $f \colon X\to X$ be a continuous map on a (typically compact) metric space $X$. We say $(X,f)$ is a (discrete) \textit{dynamical system}. A sequence $(x_i)$ in $X$, which might be finite, infinite or bi-infinite, is called a \emph{$\delta$\textit{-pseudo-orbit}} provided $d(f(x_{i}),x_{i+1}) <\delta$ for each $i$. Pseudo-orbits are obviously relevant when calculating an orbit numerically, as rounding errors mean a computed orbit will in fact be a pseudo-orbit. The (finite or infinite) sequence $(y_i)$ in $X$ is said to \emph{$\epsilon$-shadow} the $(x_i)$ provided $d(y_i,x_i)<\epsilon$ for all indices $i$. We then say that the system has \textit{shadowing}, or \textit{the pseudo-orbit tracing property}, if pseudo-orbits are shadowed by true orbits (see Section \ref{ShadowingTypes} for precise definitions). 

Whilst shadowing is clearly important when modelling a system numerically (for example \cite{Corless, Pearson}), it is also been found to have theoretical importance; for example, Bowen \cite{bowen-markov-partitions} used shadowing implicitly as a key step in his proof that the nonwandering set of an Axiom A diffeomorphism is a factor of a shift of finite type. Since then it has been studied extensively, in the setting of numerical analysis \cite{Corless,CorlessPilyugin,Pearson}, as an important factor in stability theory \cite{Pilyugin, robinson-stability,WaltersP}, in understanding the structure of $\omega$-limit sets and Julia sets \cite{BarwellGoodOprochaRaines, BarwellMeddaughRaines2015, BarwellRaines2015, Bowen, MeddaughRaines}, and as a property in and of itself \cite{Coven, GoodMeddaugh2018, LeeSakai, Nusse, Pennings, Pilyugin,Sakai2003}. 

Various other notions of shadowing have since been studied including, for example, ergodic, thick and Ramsey shadowing \cite{brian-oprocha, bmr, Dastjerdi, Fakhari, Oprocha2016}, limit shadowing \cite{BarwellGoodOprocha, GoodOprochaPuljiz2019,
Pilyugin2007}, $s$-limit shadowing \cite{BarwellGoodOprocha,GoodOprochaPuljiz2019, LeeSakai}, orbital shadowing \cite{GoodMeddaugh2016, Pilyugin2007, PiluginRodSakai2002}, and inverse shadowing \cite{CorlessPilyugin, Lee}. 

In the course of showing that systems with shadowing are built up from shifts of finite type, the first author and Meddaugh  \cite{GoodMeddaugh2018} show that an inverse limit of systems with shadowing has shadowing and that factor maps which almost lift pseudo-orbits (see below for a definition) also preserve shadowing. A continuous function map on a compact metric space $f\colon X\to X$ induces a continuous map $2^f$ on the hyperspace of closed subsets of $X$ with the Hausdorff metric. In  \cite{GoodFernandez} it is shown that $2^f$ has shadowing if and only if $f$ has shadowing.  
It is a natural question, therefore, to ask under operations on dynamical systems which notions of shadowing are preserved.  In this paper we systematically address this question for various notions of shadowing, namely shadowing, $h$-shadowing, eventual shadowing, orbital shadowing, first and second weak shadowing, inverse shadowing and various types of limit shadowing. For each of these shadowing types we ask:
\begin{itemize}
    \item Is it preserved in the induced hyperspatial system?
    \item Is it preserved in some, or all, induced symmetric product systems?
    \item Under what conditions is it preserved under semi-conjugacy?
    \item Does an inverse limit system comprised of systems with it also have it?
    \item Does an arbitrary product of systems exhibiting it also exhibit it?
\end{itemize}
We provide definitive answers to many of these questions, although we leave some unanswered; particular difficulties seem to arise when dealing with the limit shadowing properties. Clearly some of these questions have been asked, and answered by others. In such cases we provide references. 

In order to simplify proofs and keep the results as general as possible our setting throughout will be a compact Hausdorff space $X$. In particular this means all our results hold for compact metric spaces and the reader will lose very little assuming that all spaces are compact metric.  

The paper is arranged as follows. We begin with some preliminaries in Section \ref{SectionPreliminaries}, where amongst other things we give the definitions of uniform space, hyperspace, symmetric product, inverse limit space and product space. In Section \ref{ShadowingTypes} we provide the definitions of the shadowing types under consideration. We start with the usual metric definitions, before giving the uniform definitions which coincide with the metric ones when the underlying space is compact. Finally we follow the example of Good and Mac\'{\i}as \cite{GoodMacias} by providing definitions in terms of open covers which coincide with the uniform definitions when the space is compact Hausdorff. We then devote a section to the preservation of each of the aforementioned types of shadowing.

The table below provides a summary of our results.

\enspace

\begin{tabular}{c||c|c|c|c|c|}
 & $2^X$ & $F_n(X)$ & $\varphi$ & $\varprojlim$ & $\prod$  \\
\hline
\hline
shadowing&\y&\x \, (but \y for $F_2(X)$) & iff ALP &\y&\y\\
\hline

h-shad.&\y&\x \, (but \y for $F_2(X)$)&?&?&\y*\\
\hline
eventual shad.&\x&\x \, (but \y for $F_2(X)$)&iff eALP&\y&\y\\
\hline
orbital shad.&\x&\x& iff oALP &\y&\x\\
\hline
strong orb. shad.&\x&\x& iff soALP &\y&\x\\
\hline
1\textsuperscript{st} weak shad. &\x&\x& iff w1ALP &\y&\x\\
\hline
2\textsuperscript{nd} weak shad.&\y&\y&\y&\y&\y\\
\hline
limit shad. &?&\x \, (but \y for $F_2(X)$)&iff ALAP&?&\y\\
\hline
s-limit shad.&?&\x \, (but \y for $F_2(X)$)&iff ALA$\epsilon$P &?&\y\\
\hline
orb. limit shad.&\x&\x& iff oALAP&?&\x\\
\hline
inverse shad.&\y&\y&?&\y&\y\\
\hline

\end{tabular}

\enspace

KEY:
\begin{itemize}
    \item \y - ``is preserved by.''
    \item \x - ``there is a (surjective) counterexample in which it is not preserved.''
    \item \y* - ``iff all but a finite number of the component systems are surjective.''
\end{itemize}


We denote by $\mathbb{Z}$ the set of all integers; the set of positive integers $1,2,3,4,\ldots$ is denoted by $\mathbb{N}$ whilst $\omega \coloneqq\mathbb{N}\cup\{0\}$. The set of all real numbers is denoted $\mathbb{R}$, whilst $\mathbb{Q}$ denotes the set of rational numbers.
\section{Preliminaries}\label{SectionPreliminaries}

\subsection{Dynamical systems}
A \textit{dynamical system} is a pair $(X,f)$ consisting of a topological space $X$ and a continuous function $f\colon X \to X$. We say that the \textit{orbit} of $x$ under $f$ is the set of points $\{x, f(x), f^2(x), \ldots\}$; we denote this set by $\Orb_f(x)$. We define the $\omega$\textit{-limit set} of a sequence $(x_i)_{i \geq 0}$ in $X$ as the set 
\[\omega((x_i)_{i \geq 0})= \bigcap_{N \in \mathbb{N}}\overline{\{x_n \mid n >N\}}.\]
For a point $x \in X$, we define the $\omega$\textit{-limit set} of $x$ under $f$, denoted $\omega(x)$, to be $\omega$-limit set of its orbit sequence: $\left(f^n(x)\right)_{n \in \mathbb{N}}$. Formally
\[\omega(x)= \bigcap_{N \in \mathbb{N}}\overline{\{f^n(x) \mid n >N\}}.\]
If $X$ is compact then $\omega(x) \neq \emptyset$ for any $x \in X$ by Cantor's intersection theorem. 

We say a dynamical system $(X,f)$ is \textit{onto} or \textit{surjective} if $f \colon X \to X$ is a surjection. We do not assume, unless stated, that a dynamical system is necessarily onto. However, since surjective dynamical systems are usually the more interesting from a dynamics viewpoint, we ensure that every counterexample we construct in this paper is surjective (aside from in Example \ref{ExampleHShadNotSurjectiveProduct} where surjectivity is under examination).

If $(X,f)$ and $(Y,g)$ are dynamical systems we call a continuous surjection $\varphi \colon X \to Y$ a \textit{factor map} if 
\[\varphi \circ f = g \circ \phi. \]



\subsection{Uniform spaces}
Let $X$ be a nonempty set and $A \subseteq X \times X$. Let $A^{-1}=\{(y,x) \mid (x,y) \in A\}$; we call this the \textit{inverse} of $A$. The set $A$ is said to be \textit{symmetric} if $A=A^{-1}$. For any $A_1, A_2 \subseteq X \times X$ we define the composite $A_1 \circ A_2$ of $A_1$ and $A_2$ as 
\[A_1 \circ A_2 = \{(x,z) \mid \exists y\in X : (x,y) \in A_1, (y,z) \in A_2\}.\]
For any $n \in \mathbb{N}$ and $A \subseteq X \times X$ we denote by $nA$ the $n$-fold composition of $A$ with itself, i.e. 
\[nA=\underbrace{A\circ A\circ A \cdots A}_{n \text{ times}}.\]
The \textit{diagonal} of $X \times X$ is the set $\Delta= \{(x,x) \mid x \in X\}$. A subset $A \subseteq X\times X$ is called an \textit{entourage} if $A \supseteq \Delta$.

\begin{definition} A \textit{uniformity} $\mathscr{U}$ on a set $X$ is a collection of entourages of the diagonal such that the following conditions are satisfied.

\begin{enumerate}[label=\alph*.]
\item $E_1, E_2 \in \mathscr{U} \implies E_1 \cap E_2 \in \mathscr{U}$.
\item $E \in \mathscr{U}, E \subseteq D \implies D \in \mathscr{U}$.
\item $E \in \mathscr{U} \implies D \circ D \subseteq E$ for some $D \in \mathscr{U}$.
\item $E \in \mathscr{U} \implies D^{-1}\subseteq E$ for some $D \in \mathscr{U}$.
\end{enumerate}

\end{definition}

We call the pair $(X, \mathscr{U})$ a {\em uniform space}. We say $\mathscr{U}$ is {\em separating} if $\bigcap_{E \in \mathscr{U}} E = \Delta$; in this case we say $X$ is {\em separated}. A subcollection $\mathscr{V}$ of $\mathscr{U}$ is said to be a {\em base} for $\mathscr{U}$ if for any $E \in\mathscr{U}$ there exists $D \in \mathscr{V}$ such that $D \subseteq E$. Clearly any base $\mathscr{V}$ for a uniformity will have the following properties:
\begin{enumerate}
\item $E_1, E_2 \in \mathscr{U} \implies$ there exists $D \in \mathscr{V}$ such that $D \subseteq E_1 \cap E_2 $.
\item $E \in \mathscr{U} \implies D \circ D \subseteq E$ for some $D \in \mathscr{V}$.
\item $E \in \mathscr{U} \implies D^{-1}\subseteq E$ for some $D \in \mathscr{V}$.
\end{enumerate}
If $\mathscr{U}$ is separating then $\mathscr{V}$ will satisfy $\bigcap_{E \in \mathscr{V}} E = \Delta$. A {\em subbase} for $\mathscr{D}$ is a subcollection such that the collection of all finite intersections from said subcollection form a base.

\begin{remark}\label{RemarkSymFormBase} It is easy to see that the symmetric entourages of a uniformity $\mathscr{U}$ form a base for said uniformity.
\end{remark}


For an entourage $E \in \mathscr{U}$ and a point $x \in X$ we define the set $B_E(x)= \{y \in X \mid (x,y) \in E\}$; we refer to this set as the $E$-\textit{ball about} $x$. This naturally extends to a subset $A \subseteq X$; $B_E(A)= \bigcup_{x \in A}B_E(x)$; in this case we refer to the set $B_E(A)$ as the $E$-\textit{ball about} $A$. We emphasise that (see \cite[Section 35.6]{Willard}):
\begin{itemize}
\item For all $x \in X$, the collection $\mathscr{B}_x \coloneqq \{ B_E(x) \mid E \in \mathscr{U} \}$ is a neighbourhood base at $x$, making $X$ a topological space. The same topology is produced if any base $\mathscr{V}$ of $\mathscr{U}$ is used in place of $\mathscr{U}$.
\item The topology is Hausdorff if and only if $\mathscr{U}$ is separating.
\end{itemize}



For a compact Hausdorff space $X$ there is a unique uniformity $\mathscr{U}$ which induces the topology and the space is metric if the uniformity has a countable base (see \cite[Chapter 8]{Engelking}). For a metric space, a natural base for the uniformity would be the $1/2^n$ neighbourhoods of the diagonal. 

\subsection{Hyperspaces}

For a uniform space $(X,\mathscr{U})$, set \[2^X=\{A \subseteq X \mid A \text{ is compact and nonempty}\}.\] Let $\mathscr{B}_\mathscr{U}$ be the family of all sets
\[2^V \coloneqq \{(A,A^\prime) \mid A \subseteq B_V(A^\prime) \text{ and } A^\prime \subseteq B_V(A) \}, \, V \in \mathscr{U}.\]
The uniformity on the set $2^X$ generated by the base $\mathscr{B}_\mathscr{U}$ is denoted $2^\mathscr{U}$. If $X$ is a compact Hausdorff space then $2^X$ forms a compact Hausdorff topological space with the topology, known as the Vietoris topology, induced by this uniformity. If $X$ is a compact metric space then $2^X$ is a compact metric space when equipped with the \textit{Hausdorff metric}: \[d_H (A,A^\prime)= \inf \{\epsilon>0 \colon A \subseteq B_\epsilon (A^\prime) \text{ and } A^\prime \subseteq B_\epsilon (A)\}. \]
The topology generated by this metric is the Vietoris topology (see for example \cite{MaciasTop}).

For $n \geq 2$, we denote by $F_n(X)$ the $n$-\textit{fold symmetric product of} $X$, i.e.
\[F_n(X)=\{A \in 2^X \mid A \text{ contains at most }n \text{ points.}\}.\]
$F_n(X)$ is a compact Hausdorff space with the subspace topology from $2^X$.

If $X$ is compact Hausdorff and $f\colon X\to X$, then the image of a closed set $C$ under $f$ is again a closed subset of $X$.  Therefore, a given
dynamical system $(X,f)$ gives rise to an induced system $(2^X,2^f)$ on the hyperspace by $2^f\colon C\mapsto f(C)=\{f(x):x\in C\}$. The restriction of $2^f$ to $F_n$ is denoted $f_n$.

\subsection{Products and inverse limits}

Let $\{X_\lambda:\lambda\in\Lambda\}$ be a family of topological spaces. Given the product 
\[ \prod_{\lambda \in \Lambda} X_\lambda = \left\{ ( x_\lambda )_{\lambda \in \Lambda} \mid \forall \lambda \in \Lambda, x_\lambda \in X_\lambda \right\},\]
for each $\eta \in \Lambda$ the projection  $\pi_\eta \colon \prod_{\lambda \in \Lambda} X_\lambda \to X_\eta$ is defined by $\pi_\eta(( x_\lambda )) =x_\eta$.  As usual the 
\textit{Tychonoff product topology} on $\prod_{\lambda \in \Lambda} X_\lambda$ is the topology generated by basic open sets of the form
\[ \bigcap_{i=1}^n \pi_{\lambda_i}^{-1}(U_{\lambda_i}),\]
for some $n\in \mathbb{N}$ and open $U_{\lambda_i}$ in $X_{\lambda_i}$.



If, in the above, each space $X_\lambda$ is compact Hausdorff with uniformity $\mathscr{U}_\lambda$, then the following is a basic entourage in the uniformity on the product space:
\[ \prod _{\lambda \in \Lambda} E_\lambda, \]
where
$E_\lambda \in \mathscr{U}_\lambda$ for all $\lambda$ and $E_\lambda = X_\lambda \times X_\lambda$ for all but finitely many $\lambda$. The set of all such entourages forms a base for the uniformity on the product space.

Given a collection of dynamical systems $(X_\lambda,f_\lambda)$ we refer to the product system $\left(\prod_{\lambda \in \Lambda} X_\lambda,f\right)$, where $f$ is the induced map given by $f(( x_\lambda )_{\lambda \in \Lambda})=( f_\lambda(x_\lambda))_{\lambda \in \Lambda}$. It is straightforward to check that $f$ is continuous (and onto) if and only if each $f_\lambda$ is continuous (and onto).

\medskip

Let  $(\Lambda, \leq)$ be a directed set. For each $\lambda \in \Lambda$, let $X_\lambda$ be a compact Hausdorff space and, for each pair $\lambda \leq \eta$, let $g_{\lambda}^{\eta} \colon X_\eta \to X_\lambda$ be a continuous map. Suppose further that $g^\lambda_\lambda$ is the identity and that
   for all $\lambda \leq \eta \leq \nu$, $g_{\lambda}^{\nu}=g_{\lambda}^{\eta}\circ g_{\eta}^{\nu}$.

\begin{definition}\label{inverselim}
Let $(\Lambda, \leq)$ be a directed set. For each $\lambda \in \Lambda$, let $( X_\lambda,f_\lambda)$ be a surjective dynamical system on a compact Hausdorff space and, for each pair $\lambda, \eta$, with $\lambda \leq \eta$, let $g_{\lambda}^{\eta}: X_\eta \to X_\lambda$ be a continuous (not necessarily surjective) map. Suppose further that 
\begin{enumerate}
     \item $g^\lambda_\lambda$ is the identity map for all $\lambda \in \Lambda$, and
    \item for all triplets $\lambda \leq \eta \leq \nu$, $g_{\lambda}^{\nu}=g_{\lambda}^{\eta}\circ g_{\eta}^{\nu}$, and
    \item for all pairs $\lambda \leq \eta$, $f_\lambda \circ g_{\lambda}^{\eta}=g_{\lambda}^{\eta} \circ f_\eta$ (i.e. that $g_{\lambda}^{\eta}$ is a semiconjugacy).
\end{enumerate}
Then the \textit{inverse limit} of $(X_\lambda, g_{\lambda}^{\eta})$ is the compact Hausdorff space
\[\varprojlim\{ X_\lambda, g_{\lambda}^{\eta} \}=\{ ( x_\lambda ) \in \prod X_\lambda \mid \forall \lambda, \eta \text{ with } \lambda \leq \eta, x_\lambda=g_{\lambda}^{\eta}(x_\eta)\},\]
with topology inherited as a subspace of the product $\prod X_\lambda$. Moreover, the maps $f_\lambda$ induce a continuous map 
\begin{align*}
    f\colon \varprojlim\{ X_\lambda, g_{\lambda}^{\eta} \}&\to \varprojlim\{ X_\lambda, g_{\lambda}^{\eta} \},\\[3pt]
    ( x_\lambda )_{\lambda \in \Lambda} &\mapsto
    ( f_\lambda(x_\lambda) )_{\lambda \in \Lambda},
\end{align*}
resulting in the inverse system 
$\left((X_\lambda,f_\lambda),g_{\lambda}^{\eta}\right)=\left( \varprojlim \{X_\lambda, g_{\lambda}^{\eta} \}, f \right)$. 
\end{definition}

Given a system $(X,f)$, a frequently studied inverse limit system is that of the shift map $\sigma$ taking  $(x_0,x_1, x_2,\dots)$ to $(x_1,x_2, x_3\dots)$ acting as a homeomorphism on the inverse limit space $\varprojlim(X,f)=\{(x_i):x_i=f(x_{i+1}), 0\leq i\}$. Notice Definition \ref{inverselim} subsumes this definition. Given a dynamical system $(X,f)$ we will refer to the system $(\varprojlim(X,f), \sigma)$ as the \textit{standard inverse limit associated with $(X,f)$}. Note that the preservation of shadowing properties in the standard inverse limit system has been studied by various authors \cite{Barzanouni, ChenLiangShi, GoodOprochaPuljiz2019}.

\begin{definition} We say that an inverse system 
is \textit{surjective} provided that for any $\lambda  \in \Lambda$ and any $\gamma\geq \lambda$, $g^\gamma _\lambda(X_\gamma) =X_\lambda$. We say that the system is \textit{Mittag-Leffler}, provided that for all $\lambda \in \Lambda$ there exists $\gamma \geq \lambda$ such that for all $\eta \geq \gamma$, we have $g^\gamma _\lambda(X_\gamma) =g^\eta _\lambda(X_\eta)$. For such a $\lambda$ and $\gamma$ we say $\gamma$ witnesses the \textit{Mittag-Leffler condition} with respect to $\lambda$.
\end{definition}
Clearly every surjective inverse system is also Mittag-Leffler. A useful fact about Mittag-Leffler systems is that if $\gamma$ witnesses the condition with respect to $\lambda$ and $x \in g^\gamma _\lambda (X_\gamma) \subseteq X_\lambda$ then $\pi_\lambda ^{-1}(x) \cap \varprojlim \{X_\lambda, g_{\lambda}^{\eta} \} \neq \emptyset$.

\section{Shadowing types}\label{ShadowingTypes}

\subsection{Shadowing in metric spaces} Let $(X,f)$ be a dynamical system where $X$ is a metric space. 

\begin{definition}
A sequence $( x_i ) _{i \in \omega}$ in $X$ is said to be a $\delta$\textit{-pseudo-orbit} for some $\delta>0$ if $d(f(x_i),x_{i+1})<\delta$ for each $i \in \omega$.

We say $( x_i ) _{i \in \omega}$ is an \textit{asymptotic pseudo-orbit} provided that $$\lim_{i \rightarrow \infty} d(f^i(x_i), x_{i+1}) =0.$$

We say $( x_i ) _{i \in \omega}$ is an \textit{asymptotic $\delta$-pseudo-orbit} if it is both a $\delta$-pseudo-orbit and an asymptotic pseudo-orbit.
\end{definition}

\begin{definition}
A point $z \in X$ is said to \textit{$\epsilon$-shadow} a sequence $( x_i ) _{i \in \omega}$ for some $\epsilon>0$ if $d(x_i,f^i(z))<\epsilon$ for each $i \in \omega$. It \textit{asymptotically shadows} the sequence if $\lim_{i \rightarrow \infty} d(x_i, f^i(z)) =0$. Finally it \textit{asymptotically $\epsilon$-shadows} the sequence if it both $\epsilon$-shadows and asymptotically shadows it.
\end{definition}

\begin{definition}
The dynamical system $(X, f)$ is said to have \textit{shadowing} if for any $\epsilon >0$ there exists $\delta >0$ such that every $\delta$-pseudo orbit is $\epsilon$-shadowed.
\end{definition}

\begin{definition}
A system $(X,f)$ has the \textit{eventual shadowing} property provided that for all $\epsilon>0$ there exists $\delta>0$ such that for each $\delta$-pseudo-orbit $( x_i)_{i \in \omega}$, there exists $z \in X$ and $N \in \mathbb{N}$ such that $d(f^i(z), x_i) < \epsilon$ for all $i \geq N$.
\end{definition}

\begin{definition}
The system $(X, f)$ is said to have \textit{h-shadowing} if for any $\epsilon >0$ there exists $\delta >0$ such that for every finite $\delta$-pseudo orbit $\{x_0,x_1,x_2, \ldots, x_m\}$ there exists $y \in X$ such that $d(f^i(y), x_i) < \epsilon$ for all $i <m$ and $f^m(y)=x_m$.
\end{definition}

\begin{definition}
The system $(X, f)$ is said to have \textit{limit shadowing} if every asymptotic pseudo-orbit is asymptotically shadowed.
\end{definition}

\begin{definition}
The system $(X, f)$ is said to have \textit{s-limit shadowing} if for any $\epsilon >0$ there exists $\delta >0$ such that the following two conditions hold: 
\begin{enumerate}
\item every $\delta$-pseudo orbit is $\epsilon$-shadowed, and
\item every asymptotic $\delta$-pseudo orbit is asymptotically $\epsilon$-shadowed.
\end{enumerate}
\end{definition}


\begin{definition}
The system $(X,f)$ has the \textit{orbital shadowing} property if for all $\epsilon>0$, there exists $\delta>0$ such that for any $\delta$-pseudo-orbit $( x_i)_{i \in \omega}$, there exists a point $z$ such that 

\[d_H\left(\overline{\{x_i\}_{i\in\omega}}, \overline{\{f^i(z)\}_{i\in\omega}}\right) <\epsilon.\]
\end{definition}

\begin{definition}
The system $(X,f)$ has the \textit{strong orbital shadowing} property if for all $\epsilon>0$, there exists $\delta>0$ such that for any $\delta$-pseudo-orbit $(x_i)_{i \in \omega}$, there exists a point $z$ such that, for all $N \in \omega$,

\[d_H\left(\overline{\{x_{N+i}\}_{i\in\omega}}, \overline{\{f^{N+i}(z)\}_{i\in\omega}}\right) <\epsilon.\]
\end{definition}

\begin{definition}\label{DefnAsymOrbShad}
The system $(X,f)$ has the \textit{asymptotic orbital shadowing} property if for any asymptotic pseudo-orbit $(x_i)_{i\geq 0}$ there exists a point $x\in X$ such that for any $\epsilon>0$ there exists $N \in \mathbb{N}$ such that
	\begin{align*}
		d_H(\overline{\{x_{N+i}\}_{i\geq 0}},\overline{\{f^{N+i}(x)\}_{i\geq 0}})<\epsilon.
	\end{align*}
\end{definition}

This is equivalent (see \cite[Theorem 22]{GoodMeddaugh2016}) to the following definition of orbital limit shadowing studied by Pilyugin and others \cite{Pilyugin2007}.

\begin{definition}
The system $(X,f)$ has the \textit{orbital limit shadowing} property if given any asymptotic pseudo-orbit  $(x_i)_{i\geq 0}\subseteq X$, there exists a point $x\in X$ such that
	\begin{align*}
		\omega((x_i)_{i\geq 0})=\omega(x).
	\end{align*}
\end{definition}

\begin{definition} 
The system $(X,f)$ has the \textit{first weak shadowing} property if for all $\epsilon>0$, there exists $\delta>0$ such that for any $\delta$-pseudo-orbit $( x_i)$, there exists a point $z$ such that

\[\{x_i\}_{i\in\omega} \subseteq B_\epsilon\left( \Orb(z)\right).\]
\end{definition}

\begin{definition}
The system $(X,f)$ has the \textit{second weak shadowing} property if for all $\epsilon>0$, there exists $\delta>0$ such that for any $\delta$-pseudo-orbit $( x_i)_{i \in \omega}$, there exists a point $z$ such that

\[\Orb(z) \subseteq B_\epsilon\left( \{x_i\}_{i\in\omega}\right).\]
\end{definition}



Let $X$ be a compact metric space, and let $f \colon X \to X$ be a continuous onto function. Let $X^\omega$ denote the product space of all infinite sequences; note that this is compact metric. Then, for any given $\delta>0$, let $\Phi_f(\delta) \subseteq X^\omega$ be the set of all $\delta$-pseudo-orbits. We call a mapping $\varphi \colon X \to \Phi_f(\delta)$ such that, for each $x \in X$, $\varphi(x)_0=x$, a $\delta$-method for $f$ where $\varphi(x)_k$ is used to denote the $k$\textsuperscript{th} term in the sequence $\varphi(x)$. We denote by $\mathcal{T}_0(f, \delta)$ the set of all $\delta$-methods.

\begin{definition}
Let $f \colon X \to X$ be a  continuous onto function. We say that $f$ experiences \textit{inverse shadowing with respect to the class $\mathcal{T}_0$} (henceforth simply, \textit{inverse shadowing}) if, for any $\epsilon>0$ there exists $\delta>0$ such that for any $x \in X$ and any $\varphi \in \mathcal{T}_0$ there exists $y \in X$ such that $x$ $\epsilon$-shadows $\varphi(y)$; i.e.
\[ \forall k \in \omega, d(\varphi(y)_k, f^k(x))< \epsilon.\]
\end{definition}

\subsection{Shadowing in uniform spaces}
Let $(X, f)$ be a dynamical system where $X$ is a uniform space with uniformity $\mathscr{U}$. The definitions below coincide with their corresponding ones in the previous subsection when the underlying space $X$ is compact metric. 

\begin{definition} 
A sequence $( x_i ) _{i \in \omega}$ is said to be a \textit{$D$-pseudo-orbit} for some $D \in \mathscr{U}$ if $(f(x_i),x_{i+1})\in D$ for each $i \in \omega$.

We say $( x_i ) _{i \in \omega}$ is an \textit{asymptotic pseudo-orbit} provided that for each $E \in \mathscr{U}$ there exists $N \in \mathbb{N}$ such that for all $i \geq N$ $(f^i(x_i), x_{i+1}) \in E$.

We say $( x_i ) _{i \in \omega}$ is an \textit{asymptotic $D$-pseudo-orbit} if it is both a $D$-pseudo-orbit and an asymptotic pseudo-orbit.
\end{definition}

\begin{definition}
A point $z \in X$ is said to \textit{$E$-shadow} a sequence $( x_i ) _{i \in \omega}$ for some $E \in \mathscr{U}$ if $(x_i,f^i(z))\in E$ for each $i \in \omega$. It \textit{asymptotically shadows} the sequence if for each $E \in \mathscr{U}$ there exists $N \in \mathbb{N}$ such that for all $i \geq N$ $(x_i, f^i(z)) \in E$. Finally it \textit{asymptotically $E$-shadows} the sequence if it both $E$-shadows and asymptotically shadows it.
\end{definition}

\begin{definition}
The dynamical system $(X, f)$ is said to have \textit{shadowing} if for any $E \in \mathscr{U}$ there exists $D \in \mathscr{U}$ such that every $D$-pseudo-orbit is $E$-shadowed.
\end{definition}

\begin{definition}
A system $(X,f)$ has the \textit{eventual shadowing} property provided that for all $E \in \mathscr{U}$ there exists $D \in \mathscr{U}$ such that for each $D$-pseudo-orbit $( x_i)_{i \in \omega}$, there exists $z \in X$ and $N \in \mathbb{N}$ such that $(f^i(z), x_i) \in E$ for all $i \geq N$.
\end{definition}

\begin{definition}
The system $(X, f)$ is said to have \textit{h-shadowing} if for any $E \in \mathscr{U}$ there exists $D \in \mathscr{U}$ such that for every finite $D$-pseudo orbit $\{x_0,x_1,x_2, \ldots, x_m\}$ there exists $y \in X$ such that $(f^i(y), x_i) \in E$ for all $i <m$ and $f^m(y)=x_m$.
\end{definition}

\begin{definition}
The system $(X, f)$ is said to have \textit{limit shadowing} if every asymptotic pseudo-orbit is asymptotically shadowed.
\end{definition}

\begin{definition}
The system $(X, f)$ is said to have \textit{s-limit shadowing} if for any $E\in \mathscr{U}$ there exists $D \in \mathscr{U}$ such that the following two conditions hold: 
\begin{enumerate}
\item every $D$-pseudo-orbit is $E$-shadowed, and
\item every asymptotic $D$-pseudo-orbit is asymptotically $E$-shadowed.
\end{enumerate}
\end{definition}


\begin{definition}
The system $(X,f)$ has the \textit{orbital shadowing} property if for all $E \in \mathscr{U}$, there exists $D \in \mathscr{U}$ such that for any $D$-pseudo-orbit $( x_i)_{i \in \omega}$, there exists a point $z \in X$ such that 

\[\left(\overline{\{x_i\}_{i\in\omega}}, \overline{\{f^i(z)\}_{i\in\omega}}\right) \in 2^E.\]
In this case we say $z$ $E$-orbital shadows $( x_i)_{i \in \omega}$.
\end{definition}

\begin{definition}
The system $(X,f)$ has the \textit{strong orbital shadowing} property if for all $E \in \mathscr{U}$, there exists $D \in \mathscr{U}$ such that for any $D$-pseudo-orbit $( x_i)_{i \in \omega}$, there exists a point $z \in X$ such that, for all $N \in \omega$,

\[\left(\overline{\{x_{N+i}\}_{i\in\omega}}, \overline{\{f^{N+i}(z)\}_{i\in\omega}}\right) \in 2^E.\]
In this case we say $z$ $E$-strong-orbital shadows $( x_i)_{i \in \omega}$.
\end{definition}

\begin{definition}
The system $(X,f)$ has the \textit{asymptotic orbital shadowing} property if for any asymptotic pseudo-orbit $(x_i)_{i\geq 0}$ there exists a point $x\in X$ such that for any $E \in \mathscr{U}$ there exists $N \in \mathbb{N}$ such that
	\begin{align*}
		(\overline{\{x_{N+i}\}_{i\geq 0}},\overline{\{f^{N+i}(x)\}_{i\geq 0}}) \in 2^E.
	\end{align*}
\end{definition}

\begin{definition}
The system $(X,f)$ has the \textit{orbital limit shadowing} property if given any asymptotic pseudo-orbit  $(x_i)_{i\geq 0}\subseteq X$, there exists a point $x\in X$ such that
	\begin{align*}
		\omega((x_i)_{i\geq 0})=\omega(x).
	\end{align*}
\end{definition}

\begin{definition}
The system $(X,f)$ has the \textit{first weak shadowing} property if for all $E \in \mathscr{U}$, there exists $D \in \mathscr{U}$ such that for any $D$-pseudo-orbit $( x_i)_{i \in \omega}$, there exists a point $z$ such that

\[\{x_i\}_{i\in\omega} \subseteq B_E\left( \Orb(z)\right).\]
\end{definition}

\begin{definition}
The system $(X,f)$ has the \textit{second weak shadowing} property if for all $E \in \mathscr{U}$, there exists $D \in \mathscr{U}$ such that for any $D$-pseudo-orbit $( x_i)_{i \in \omega}$, there exists a point $z$ such that

\[\Orb(z) \subseteq B_E\left( \{x_i\}_{i\in\omega}\right).\]
\end{definition}



Let $X$ be a compact Hausdorff space, and let $f \colon X \to X$ be a continuous onto function. Let $X^\omega$ denote the product space of all infinite sequences. Then, for any given $D \in \mathscr{U}$, let $\Phi_f(D) \subseteq X^\omega$ be the set of all $D$-pseudo-orbits. We call a mapping $\varphi \colon X \to \Phi_f(D)$ such that, for each $x \in X$, $\varphi(x)_0=x$, a $D$-method for $f$ where $\varphi(x)_k$ is used to denote the $k$\textsuperscript{th} term in the sequence $\varphi(x)$. We denote by $\mathcal{T}_0(f, D)$ the set of all $D$-methods.

\begin{definition}
Let $f \colon X \to X$ be a  continuous onto function. We say that $f$ experiences \textit{inverse shadowing with respect to the class $\mathcal{T}_0$} (henceforth simply, \textit{inverse shadowing}) if, for any $E \in \mathscr{U}$ there exists $D \in \mathscr{U}$ such that for any $x \in X$ and any $\varphi \in \mathcal{T}_0$ there exists $y \in X$ such that $x$ $E$-shadows $\varphi(y)$; i.e.
\[ \forall k \in \omega, (\varphi(y)_k, f^k(x))\in E.\]
\end{definition}

\begin{remark}\label{RemarkAssumeEntouragesAreSymmetric} It follows from Remark \ref{RemarkSymFormBase} that, without loss of generality, we may assume all entourages referred to in the above definitions are symmetric. Throughout what follows we will make this assumption.
\end{remark}

\subsection{Shadowing with open covers}
Let $X$ be a topological space and $f \colon X \to X$ a continuous function. The definitions below coincide with their corresponding ones in the previous subsection when the underlying space $X$ is compact Hausdorff.

\begin{definition}
A sequence $( x_i ) _{i \in \omega}$ is said to be a \textit{$\mathcal{U}$-pseudo-orbit} for some open cover $\mathcal{U}$ if for any $i \in \omega$ there exists $U \in \mathcal{U}$ with $f(x_i),x_{i+1} \in U$.
\end{definition}

\begin{definition}
A point $z \in X$ is said to \textit{$\mathcal{U}$-shadow} a sequence $( x_i ) _{i \in \omega}$ for some open cover $\mathcal{U}$ if for any $i \in \omega$ there exists $U \in \mathcal{U}$ with $x_i,f^i(z) \in U$. We say $z \in X$ \textit{eventually-$\mathcal{U}$-shadows} a sequence $( x_i ) _{i \in \omega}$ for some open cover $\mathcal{U}$ if there exists $N \in \mathbb{N}$ such that for any $i \geq N$ there exists $U \in \mathcal{U}$ with $x_i,f^i(z) \in U$.
\end{definition}

\begin{definition}
The dynamical system $(X, f)$ is said to have \textit{shadowing} if for any finite open cover $\mathcal{U}$ there exists a finite open cover $\mathcal{V}$ such that every $\mathcal{V}$-pseudo-orbit is $\mathcal{U}$-shadowed. 
\end{definition}

\begin{definition}
The dynamical system $(X, f)$ is said to have \textit{eventual shadowing} if for any finite open cover $\mathcal{U}$ there exists a finite open cover $\mathcal{V}$ such that every $\mathcal{V}$-pseudo-orbit is eventually-$\mathcal{U}$-shadowed. 
\end{definition}

\begin{definition}
The dynamical system $(X, f)$ is said to have \textit{h-shadowing} if for any finite open cover $\mathcal{U}$ there exists a finite open cover $\mathcal{V}$ such that for any finite $\mathcal{V}$-pseudo-orbit $\{x_0,x_1,x_2, \ldots, x_m\}$ there exists $y \in X$ such that for any $i<m$ there exists $U \in \mathscr{U}$ with $f^i(y), x_i \in U$ and $f^m(y)=x_m$.
\end{definition}

\begin{definition}
The system $(X,f)$ has the \textit{orbital shadowing} property if for any finite open cover $\mathcal{U}$ there exists a finite open cover $\mathcal{V}$ such that for any $\mathcal{V}$-pseudo-orbit $( x_i)_{i \in \omega}$ there exists a point $z \in X$ such that 
\[\forall y \in \overline{\Orb(z)} \, \exists U \in \mathcal{U} \, \exists y^\prime \in \overline{\{x_i \mid i \in \omega\}} : y,y^\prime \in U,\]
and
\[\forall y^\prime \in \overline{\{x_i \mid i \in \omega\}} \, \exists U \in \mathcal{U} \, \exists y \in \overline{\Orb(z)} : y,y^\prime \in U.\]
\end{definition}

\begin{definition}\label{DefnOpenCoverStrongOrbShad}
The system $(X,f)$ has the \textit{strong orbital shadowing} property if for any finite open cover $\mathcal{U}$ there exists a finite open cover $\mathcal{V}$ such that for any $\mathcal{V}$-pseudo-orbit $( x_i)_{i \in \omega}$ there exists a point $z \in X$ such that for any $N \in \omega$ 
\[\forall y \in \overline{\Orb(f^{N}(z))} \, \exists U \in \mathcal{U} \, \exists y^\prime \in \overline{\{x_{N+i} \mid i \in \omega\}} : y,y^\prime \in U,\]
and
\[\forall y^\prime \in \overline{\{x_{N+i} \mid i \in \omega\}} \, \exists U \in \mathcal{U} \, \exists y \in \overline{\Orb(f^N(z))} : y,y^\prime \in U.\]
\end{definition}

\begin{definition}
The system $(X,f)$ has the \textit{first weak shadowing} property if for any finite open cover $\mathcal{U}$ there exists a finite open cover $\mathcal{V}$ such that for any $\mathcal{V}$-pseudo-orbit $( x_i)_{i \in \omega}$ there exists a point $z \in X$ such that 
\[\forall y \in \overline{\{x_i \mid i \in \omega\}} \, \exists U \in \mathcal{U} \, \exists y^\prime \in \overline{\Orb(z)} : y,y^\prime \in U.\]

\end{definition}

\begin{definition} 
The system $(X,f)$ has the \textit{second weak shadowing} property if for any finite open cover $\mathcal{U}$ there exists a finite open cover $\mathcal{V}$ such that for any $\mathcal{V}$-pseudo-orbit $( x_i)_{i \in \omega}$ there exists a point $z \in X$ such that 
\[\forall y \in \overline{\Orb(z)} \, \exists U \in \mathcal{U} \, \exists y^\prime \in \overline{\{x_i \mid i \in \omega\}} : y,y^\prime \in U.\]
\end{definition}

Let $X$ be a compact Hausdorff space, and let $f \colon X \to X$ be a continuous onto function. Let $X^\omega$ denote the product space of all infinite sequences. Then, for any given finite open cover $\mathcal{U}$, let $\Phi_f(\mathcal{U}) \subseteq X^\omega$ be the set of all $\mathcal{U}$-pseudo-orbits. We call a mapping $\varphi \colon X \to \Phi_f(\mathcal{U})$ such that, for each $x \in X$, $\varphi(x)_0=x$, a $\mathcal{U}$-method for $f$ where $\varphi(x)_k$ is used to denote the $k$\textsuperscript{th} term in the sequence $\varphi(x)$. We denote by $\mathcal{T}_0(f, \mathcal{U})$ the set of all $\mathcal{U}$-methods.

\begin{definition}
Let $f \colon X \to X$ be a  continuous onto function. We say that $f$ experiences \textit{inverse shadowing with respect to the class $\mathcal{T}_0$} (henceforth simply, \textit{inverse shadowing}) if, for any finite open cover $\mathcal{U}$ there exists a finite open cover $\mathcal{V}$ such that for any $x \in X$ and any $\varphi \in \mathcal{T}_0$ there exists $y \in X$ such that $\varphi(y)$ $\mathcal{U}$-shadows $x$; i.e.
\[ \forall k \in \omega \, \exists U \in \mathcal U : \varphi(y)_k, f^k(x)\in U.\]
\end{definition}

For the rest of this paper, unless otherwise stated, $X$ is taken to be a compact Hausdorff space and $f \colon X \to X$ a continuous function. Similarly, unless otherwise stated, by ``dynamical system'', we are assuming the underlying phase space is compact Hausdorff.

\section{Preservation of Shadowing}
As mentioned in the introduction, Bowen \cite{bowen-markov-partitions} was one of the first to us the property of shadowing in his study of Axiom A diffeomorphisms and since then it has been both used as a tool and studied extensively in a property in its own right (see, for examples, \cite{BarwellGoodOprochaRaines, Bowen, Corless,CorlessPilyugin, Coven,GoodMeddaugh2018, LeeSakai,MeddaughRaines, Nusse,Pearson,Pennings,Pilyugin, robinson-stability, Sakai2003,WaltersP}).

Recall the following definition from the preliminaries:
the dynamical system $(X, f)$ is said to have \textit{shadowing} if for any $E \in \mathscr{U}$ there exists $D \in \mathscr{U}$ such that every $D$-pseudo-orbit is $E$-shadowed.

\subsection{Induced map on the hyperspace of compact sets}
The following theorem was proved in \cite{GoodFernandez} for compact metric systems. The proof easily generalises to compact  Hausdorff systems.
\begin{theorem}\textup{\cite[Theorem 3.4]{GoodFernandez}} Let $X$ be a compact Hausdorff space and let $f \colon X \to X$ be a continuous function. Then $(X,f)$ has shadowing if and only if $(2^X,2^f)$ has shadowing.
\end{theorem}

\subsection{Symmetric products}
In \cite{GomezIllanesMendez} the authors show that, for any $n \in \mathbb{N}$, if $f_n$ has shadowing then $f$ has shadowing. They also show that if $f$ has shadowing then $f_2$ has shadowing. However they provide an example ($z \mapsto z^2$ on the unit circle $S^1$) for which $f$ has shadowing but $f_n$ does not have shadowing for any $n \geq 3$. The following is another such example and will be recalled later.

\begin{example}\label{ExTentMapDoesNotPreserveShadToSymProducts}
Let $X$ be the closed unit interval and let $f \colon X \to X$ be the standard tent map, i.e.
 \[f(x)=\left\{\begin{array}{lll}
2x & \text{ if } & x \in [0,\frac{1}{2}]
\\2(1-x) & \text{ if } & x \in (\frac{1}{2},1] \,.
\end{array}\right.\]
Then $f$ has shadowing \textup{\cite[Example 3.5]{BarwellGoodOprochaRaines}} but $f_n$ does not have shadowing for any $n \geq 3$.

Fix $n \geq 3$. Let $c= \frac{2}{3}$. Let $\epsilon = \frac{1}{12}$ and let $\delta>0$ be given; without loss of generality $\delta<\frac{1}{12}$. Choose $y\in [0,\delta)$ such that there exists $k \in \mathbb{N}$ such that $f^k(y) =c$ and $f^i(y)< \frac{1}{2}$ for all $i <k$. Construct a $\delta$-pseudo-orbit as follows. For any $i \in \omega$ let $A_i=\{0,f^{i \mod k}(y),c\}$. It is easy to see that $(A_i)_{i \in \omega}$ is a $\delta$-pseudo-orbit. Suppose that $A \in F_n(X)$ $\epsilon$-shadows this pseudo-orbit. First observe that, since the pseudo-orbit is always a subset of the interval $[0,\frac{2}{3}]$, shadowing entails that $f^i(A) \subseteq [0,\frac{3}{4}]$ for any $i \in \omega$. Next notice that, by construction, there exists $k_0 \in \{1,\ldots k-1\}$ such that $A_{mk+k_0} \cap (\epsilon,2\epsilon] \neq \emptyset$ for all $m \in \omega$. By shadowing it follows that for any $m \in \omega$ there exists $a \in A$ such that $f^{mk+k_0}(a)\in (0,3\epsilon)$. Notice \[f_n^{-1}\left((0,3\epsilon)\right)=\left(0,\frac{3\epsilon}{2}\right)\cup \left(1-\frac{3\epsilon}{2}, 1\right)\subseteq \left(0,\frac{1}{4}\right)\cup \left(\frac{3}{4},1\right).\]
Now let $z$ be the least such element of $A \setminus \{0\}$. Let $l\in \omega$ be least such that $f^l(z) > 3\epsilon$. Let $m \in \omega$ be such that $mk+k_0>l$. Let $a \in A$ be such that $f^{mk+k_0}(a) \in (0,3\epsilon)$; notice $a\neq 0$. Since the preimage of $(0,3\epsilon)$ is a subset of $(0,\frac{1}{4})\cup (\frac{3}{4},1)$, since $f_n^i(A) \cap (\frac{3}{4},1]=\emptyset$ for all $i \in \omega$ and since $f$ is strictly increasing on $[0,\frac{1}{2})$, it follows that $a<z$, contradicting the minimality of $z$. Therefore $f_n$ does not have the shadowing property.
\end{example}

\subsection{Factor maps}
In \cite{GoodMeddaugh2018} the authors introduce the concept of factor maps which \textit{almost lift pseudo orbits}. For such maps, pseudo-orbits in the codomain system roughly correlate to pseudo orbits in the domain system - hence they `almost lift'.

\begin{definition} Suppose $X$ and $Y$ are compact Hausdorff spaces, $f\colon X \to X$ and $g \colon Y \to Y$ are continuous. A factor map $\varphi \colon (X,f) \to (Y,g)$ \textit{almost lifts pseudo-orbits (ALP)} if for every $V \in \mathscr{U}_Y$ and every $D \in \mathscr{U}_X$ there exists $W \in \mathscr{U}_Y$ such that for every $W$-pseudo-orbit $(y_i)_{i \in \omega}$ in $Y$, there exists a $D$-pseudo-orbit $(x_i)_{i \in \omega}$ in $X$ such that $(\varphi(x_i),y_i) \in V$ for all $i \in \omega$. 
\end{definition}

If $X$ and $Y$ are compact metric spaces, then $\varphi$ is ALP if and only if for all $\epsilon>0$ and $\eta>0$, there exists $\delta>0$ such that if $(y_i)_{i \in \omega}$ is a $\delta$-pseudo-orbit in $Y$, there exists an $\eta$-pseudo-orbit $(x_i)_{i\in \omega}$ in $X$ with $d(\phi(x_i),y_i)<\epsilon$.

\begin{theorem}\textup{\cite{GoodMeddaugh2018}}\label{ShadowALP}
Let $(X,f)$ and $(Y,g)$ be dynamical systems, where $X$ and $Y$ are compact Hausdorff, and let $\varphi \colon (X,f) \to (Y, g)$ be a factor map. Then the following statements hold:
\begin{enumerate}
    \item If $(X,f)$ has shadowing and $\varphi$ is an ALP map then $(Y,g)$ has shadowing.
    \item If $(Y,g)$ has shadowing then $\varphi$ is an ALP map.
\end{enumerate}
\end{theorem}
In particular it follows that a factor map preserves shadowing if and only if it is an ALP map.

\subsection{Inverse limits}
In \cite{GoodMeddaugh2018} the authors prove the following theorem.
\begin{theorem}\textup{\cite{GoodMeddaugh2018}}
Let $(X,f)$ be conjugate to a Mittag-Leffler inverse limit system comprised of maps with shadowing on compact Hausdorff spaces. Then $(X,f)$ has shadowing. 
\end{theorem}

\subsection{Tychonoff product}
The following result is folklore.
\begin{theorem}\label{ThmProductPreserveShadowing}
Let $\Lambda$ be an arbitrary index set and let $(X_\lambda, f_\lambda)$ be a system with shadowing for each $\lambda \in \Lambda$. Then the product system $(X,f)$, where $X=\prod_{\lambda \in \Lambda} X_\lambda$, has shadowing.
\end{theorem}

\section{Preservation of h-shadowing}
The property of h-shadowing was introduced in \cite{BarwellGoodOprochaRaines} and was motivated by the fact that certain systems, called shifts of finite type, which are fundamental in the study of shadowing (see \cite{GoodMeddaugh2018}) exhibit a stronger form of shadowing, i.e. h-shadowing, which coincides with the usual form for shift systems but is distinct in general (see \cite[Example 6.4]{BarwellGoodOprocha}).

Recall the definition from Section \ref{ShadowingTypes}: The system $(X, f)$ is said to have h-shadowing if for any $E \in \mathscr{U}$ there exists $D \in \mathscr{U}$ such that for every finite $D$-pseudo orbit $\{x_0,x_1,x_2, \ldots, x_m\}$ there exists $y \in X$ such that $(f^i(y), x_i) \in E$ for all $i <m$ and $f^m(y)=x_m$.

\begin{remark}\label{remarkHshadSurjection}
If $X$ is a perfect space (i.e. it has no isolated points) and $(X,f)$ has h-shadowing then $f$ is a surjection.
\end{remark}

\subsection{Induced map on the hyperspace of compact sets}
The following theorem was proved in \cite{GoodFernandez} for compact metric systems. Their proof generalises to give the result for compact Hausdorff systems.
\begin{theorem}\textup{\cite[Theorem 4.6]{GoodFernandez}} Let $X$ be a compact Hausdorff space and let $f \colon X \to X$ be a continuous function. Then $(X,f)$ has h-shadowing if and only if $(2^X,2^f)$ has h-shadowing.
\end{theorem}

\subsection{Symmetric products}
The following theorem is stated in \cite{GoodFernandez} for compact metric systems. The result generalises to compact Hausdorff systems.

\begin{theorem}\textup{\cite[Theorem 4.3]{GoodFernandez}} Let $X$ be a compact Hausdorff space and let $f \colon X \to X$ be a continuous function. For any $n \geq 2$, if $(F_n(X), f_n)$ has h-shadowing then $(X,f)$ has h-shadowing.
\end{theorem}

\begin{theorem}
Let $X$ be a compact Hausdorff space and let $f \colon X \to X$ be a continuous function. If $(X,f)$ has h-shadowing then $(F_2(X),f_2)$ has h-shadowing.
\end{theorem}
\begin{proof}
Let $E \in 2^\mathscr{U}$ be given. (Recall the standing assumption made in Remark \ref{RemarkAssumeEntouragesAreSymmetric}. This is, we assume, without loss of generality, that all entourages we make reference to are symmetric.) Let $E_0 \in \mathscr{U}$ be such that $2^{E_0} \subseteq E$. Let $D \in \mathscr{U}$ correspond to $E_0$ in h-shadowing for $f$. We claim $2^D$ satisfies the h-shadowing condition for $E$. Suppose that $\{A_0,A_1,\ldots A_m\}$ is a finite $2^D$-pseudo-orbit in $F_2(X)$. Write $A_i=\{x_i, y_i\}$; it is possible that, for some $i$, $x_i=y_i$. Relabelling the $x$'s and $y$'s where necessary, $\{x_0,\ldots x_m\}$ and $\{y_0,\ldots, y_m\}$ are finite $D$-pseudo-orbits in $X$. By h-shadowing there exist $x,y \in X$ such that $f^m(x)=x_m$, $f^m(y)=y_m$ and, for all $i \in \{0,\ldots, m-1\}$, $(f^i(x),x_i) \in E_0$ and $(f^i(y),y_i) \in E_0$. Write $A=\{x,y\} \in F_2(X)$. Notice $f_2^m(A)=A_m$. By the above, for each $i \in \{0,\ldots, m-1\}$, $A_i \subseteq B_{E_0}\left(f^i(A)\right)$ and $f^i(A) \subseteq B_{E_0}\left(A_i\right)$. It follows that $(f^i(A),A_i) \in 2^{E_0}$. Since $2^{E_0}\subseteq E$ we get that $(f^i(A),A_i) \in E$ for each $i \in \{0,\ldots, m-1\}$.
\end{proof}

\begin{remark}
Example \ref{ExTentMapDoesNotPreserveShadToSymProducts} shows that, in general, symmetric products do not preserve h-shadowing for $n\geq3$. The standard tent map $(X,f)$ has h-shadowing \textup{\cite[Example 3.5]{BarwellGoodOprochaRaines}} however $(F_n(X),f_n)$ does not have shadowing for any $n\geq3$. Since h-shadowing implies shadowing on compact spaces (see \cite{BarwellGoodOprochaRaines}) it follows that $(F_n(X),f_n)$ does not possess h-shadowing either.
\end{remark}

\subsection{Factor maps}

Clearly it follows from Theorem \ref{ShadowALP} that if $\varphi\colon (X,f)\to(Y,g)$ is a factor map and $Y$ has h-shadowing, then $\varphi$ is ALP. It is unclear, however, whether ALP is strong enough to preserve shadowing.

\subsection{Tychonoff product}
Recall Remark \ref{remarkHshadSurjection}: if $X$ is a perfect space and $(X,f)$ has h-shadowing then $f$ must be a surjection. For this reason an arbitrary product of dynamical systems with h-shadowing need not itself have h-shadowing (see Example \ref{ExampleHShadNotSurjectiveProduct}).

\begin{theorem}\label{thmProductHShad}
Let $\Lambda$ be an arbitrary index set and let $(X_\lambda, f_\lambda)$ be a surjective compact Hausdorff system with h-shadowing for each $\lambda \in \Lambda$. Then the product system $(X,f)$, where $X=\prod_{\lambda \in \Lambda} X_\lambda$, has h-shadowing.
\end{theorem}

\begin{proof}
Let $E \in \mathscr{U}$ be given; this entourage is refined by one of the form
\[ \prod _{\lambda \in \Lambda} E_\lambda,\]
where $E_\lambda \in \mathscr{U}_\lambda$ for all $\lambda \in \Lambda$ and $E_\lambda =X_\lambda \times X_\lambda$ for all but finitely many of the $\lambda$'s. Let $\lambda_i$, for $1\leq i \leq k$, be precisely those elements in $\Lambda$ for which $E_\lambda  \neq X_\lambda \times X_\lambda$ (if there are no such elements then we are done). By h-shadowing in each component space, there exist entourages $D_{\lambda_i} \in \mathscr{U}_{\lambda_i}$ such that every $D_{\lambda_i}$-pseudo-orbit is $E_{\lambda_i}$-\textit{h}-shadowed.
Let 
\[D \coloneqq \prod_{\lambda \in \Lambda} D_\lambda \]
where
\[D_\lambda = \left\{\begin{array}{lll}
X\times X  & \text{ if } & \forall i \, \lambda \neq \lambda_i
\\D_{\lambda_i} & \text{ if } & \exists i : \lambda=\lambda_i\,.
\end{array}\right.\]
Now let $\{x_0,x_1,\ldots, x_m\}$ be a finite $D$-pseudo-orbit. Then $\{\pi_{\lambda_i}(x_1),\pi_{\lambda_i}(x_2),\ldots, \pi_{\lambda_i}(x_m)\}$ is a $D_{\lambda_i}$-pseudo-orbit in $X_{\lambda_i}$, which is $E_{\lambda_i}$-h-shadowed by a point $z_i \in X_{\lambda_i}$. Pick a point $z \in X$ such that $\pi_{\lambda_i}(z)= z_i$ for each $1\leq i \leq k$ and $\pi_{\lambda}(f^m(z))=\pi_\lambda(x_m)$ for all $\lambda \in \Lambda$. It follows that $z$ $E$-h-shadows $\{x_0,x_1,\ldots, x_m\}$.
\end{proof}

\begin{remark}
It is easy to see that if only a finite number of the component systems involved in Theorem \ref{thmProductHShad} were not surjective the result would still hold.
\end{remark}

\begin{example}\label{ExampleHShadNotSurjectiveProduct}
For each $i \in \omega$ let $X_i = \{2\} \cup [0,1]$ and $f_i \colon X \to X$ be given by
\[f_i(x)=\left\{\begin{array}{lll}
2x & \text{ if }& x \in [0,\frac{1}{2}]
\\2(1-x) & \text{ if } & x \in (\frac{1}{2},1]
\\ 1 & \text{ if } & x=2\,.
\end{array}\right.\]
Thus, each system is comprised of the standard tent map together with an isolated point which maps to the fixed point $1$.  The standard tent map has h-shadowing \textup{\cite[Example 3.5]{BarwellGoodOprochaRaines}} and the it is obvious that the additional point in these systems does nothing to contradict that. Thus each system $(X_i,f_i)$ has h-shadowing. The product system
$(X,f)$, where 
\[X=\prod _{i \in \omega} X_i,\]
has no isolated points. However, the point given by $x_i=2$ for all $i \in \omega$ has no preimage; the system is not onto. Hence, by Remark \ref{remarkHshadSurjection}, the system $(X,f)$ does not have h-shadowing.
\end{example}

\section{Preservation of Eventual Shadowing}
Eventual shadowing was introduced in \cite{GoodMeddaugh2016} in the authors' journey to characterise when the set of $\omega$-limit sets of a system coincides with the set of closed internally chain transitive sets. As remarked upon in \cite{GoodMeddaugh2016}, the property of eventual shadowing is equivalent with the $(\mathbb{N}, \mathcal{F}_{cf})$-shadowing property of Oprocha \cite{Oprocha2016}.

Recall that a system $(X,f)$ has the \textit{eventual shadowing} property provided that for all $E \in \mathscr{U}$ there exists $D \in \mathscr{U}$ such that for each $D$-pseudo-orbit $( x_i)_{i \in \omega}$, there exists $z \in X$ and $N \in \mathbb{N}$ such that $(f^i(z), x_i) \in E$ for all $i \geq N$.

\subsection{Induced map on the hyperspace of compact sets}

\begin{theorem}\label{thmEventualShadHyperspace}Let $X$ be a compact Hausdorff space and let $f \colon X \to X$ be a continuous function. If the hyperspace system $(2^X, 2^f)$ has eventual shadowing then $(X,f)$ has eventual shadowing.
\end{theorem}
\begin{proof}
Let $E \in \mathscr{U}$. Let $D \in \mathscr{U}$ be such that $2^D$ corresponds to $2^E$ for eventual shadowing for $2^f$. Let $(x_i)_{i \in \omega}$ be a $D$-pseudo-orbit in $X$. Then $\left(\{x_i\}\right)_{i \in \omega}$ is a $2^D$-pseudo-orbit in $2^X$. By eventual shadowing there exists $A \in 2^X$ and $N \in \mathbb{N}$ such that $\left((2^f)^i(A),\{x_i\}\right) \in 2^E$ for all $i \geq N$. It follows that, for any $a \in A$, $\left(f^i(a),x_i\right) \in E$ for all $i \geq N$. Since $A \neq \emptyset$ the result holds.
\end{proof}

The following example shows that the converse to Theorem \ref{thmEventualShadHyperspace} is not true: the hyperspatial system of a system with eventual shadowing need not have eventual shadowing.

\begin{example}\label{ExNotPreserveEventualInHyperspaces}
Let $X=[-1,1]$ and let $f \colon X \to X$ be given by
\[f(x)=\left\{\begin{array}{lll}
 x^3 & \text{ if } & x \in [-1,0]
\\ 2x & \text{ if } & x \in (0,\frac{1}{2}]
\\2(1-x) & \text{ if } & t \in (\frac{1}{2},1] \,.
\end{array}\right.\]
As observed in \textup{\cite{GoodMeddaugh2016}} $f$ has eventual shadowing but not shadowing. We claim $2^f$ does not have eventual shadowing. Let $\epsilon=\frac{1}{12}$ and fix $\delta>0$; without loss of generality assume $\delta<\epsilon$. Choose a point $y \in (-1,-1+\delta)$ such that there exists $m \in \omega$ with $f^m(y)=-\frac{1}{2}$; let $k \in \omega$ be such that $f^k(y) \in (-\frac{\delta}{2}, 0]$ (notice that $m<k$). Let $p \in (0,\frac{\delta}{2})$ be periodic with period $n$ and such that there exists $n_0<n$ with $f^{n_0}(p) \in (1-\frac{\delta}{2},1)$. We may now construct a $\delta$-pseudo-orbit in $2^X$ as follows. For any $i \in \omega$:
\begin{itemize}
  \item if $i \mod (k+n) =0 $ then let $A_{i}=\{-1,y,0\}$,
    \item if $j=i \mod (k+n) \in \{1,\ldots, k-1\}$ then let $A_i=\{-1, f^{j \mod k}(y),0\}$,
    \item if $i \mod (k+n) =k$ then let $A_i=\{-1, p\}$,
    \item if $j=i \mod (k+n) \in \{k+1,\ldots, k+n-1\}$ then let $A_i=\{-1, f^{j \mod n}(p)\}$.
\end{itemize}
We claim $(A_i)_{i\in \omega}$ cannot be eventually $\epsilon$-shadowed in $2^X$. Indeed suppose $A \in 2^X$ eventually $\epsilon$-shadows this pseudo-orbit. Let $N \in \omega$ be such that \[d_H\left((2^f)^{N+i}(A), A_{N+i}\right) < \epsilon\] for all $i \in \omega$. Let $l>N$ be such that $A_l \ni f^m(y)=-\frac{1}{2}$. Then there exists $a \in A$ such that $f^l(a)\in (-\frac{1}{2}-\epsilon, -\frac{1}{2}+\epsilon)$. Now let $l_0>l$ be such that $l_0 \mod k+n =k+n_0$. Then $A_{l_0}=\{-1, f^{n_0}(p)\}$ but $f^{l_0}(a) \in (-\frac{1}{2}-\epsilon, 0)$. Since 
\[\left(-\frac{1}{2}-\epsilon, 0\right) \cap \left(B_\epsilon(-1) \cup B_\epsilon( f^{n_0}(p))\right) =\emptyset,\] 
we have a contradiction: $A$ does not eventually $\epsilon$-shadow $(A_i)_{i\in \omega}$.
\end{example}

\subsection{Symmetric products}

\begin{theorem}Let $X$ be a compact Hausdorff space and let $f \colon X \to X$ be a continuous function. For any $n \geq 2$, if $(F_n(X),f_n)$ has eventual shadowing then $(X,f)$ has eventual shadowing.
\end{theorem}
\begin{proof}
Let $E \in \mathscr{U}$. Let $D \in \mathscr{U}$ be such that $2^D\cap \left(F_n(X) \times F_n(X)\right)$ corresponds to $2^E\cap \left(F_n(X) \times F_n(X)\right)$ for eventual shadowing for $f_n$. Let $(x_i)_{i \in \omega}$ be a $D$-pseudo-orbit in $X$. Then $\left(\{x_i\}\right)_{i \in \omega}$ is a $2^D$-pseudo-orbit in $F_n(X)$. By eventual shadowing there exists $A \in F_n(X)$ and $N \in \mathbb{N}$ such that $\left(f_n ^i(A),\{x_i\}\right) \in 2^E$ for all $i \geq N$. It follows that, for any $a \in A$, $\left(f^i(a),x_i\right) \in E$ for all $i \geq N$. Since $A \neq \emptyset$ the result holds.
\end{proof}

\begin{theorem}
Let $X$ be a compact Hausdorff space and let $f \colon X \to X$ be a continuous function. If $(X,f)$ has eventual shadowing then $(X,f_2)$ has eventual shadowing.
\end{theorem}
\begin{proof}
Let $E \in 2^\mathscr{U}$ be given. Let $E_0 \in \mathscr{U}$ be such that $2^{E_0} \subseteq E$. Let $D \in \mathscr{U}$ correspond to $E_0$ in eventual shadowing for $f$. We claim $2^D \cap \left(F_2(X) \times F_2(X)\right)$ satisfies the eventual shadowing condition for $f_2$ and $E \cap \left(F_2(X) \times F_2(X)\right)$. Suppose that $(A_i)_{i \in \omega}$ is a $2^D$-pseudo-orbit in $F_2(X)$. Write $A_i=\{x_i, y_i\}$; it is possible that, for some $i$, $x_i=y_i$. Relabelling the $x$'s and $y$'s where necessary, $(x_i)_{i \in \omega}$ and $(y_i)_{i\in \omega}$ are $D$-pseudo-orbits in $X$. By eventual shadowing for $f$ there exist $x,y \in X$ and $N_1, N_2 \in \mathbb{N}$ such that for all $i \geq N_1$, $(f^i(x),x_i) \in E_0$ and for all $i \geq N_2$ and $(f^i(y),y_i) \in E_0$. Take $N = \max \{N_1, N_2\}$. Then, for all $i \geq N$, $A_i \subseteq B_{E_0}\left(f_2^i(A)\right)$ and $f_2^i(A) \subseteq B_{E_0}\left(A_i\right)$. It follows that $(f_2 ^i(A),A_i) \in 2^{E_0}\cap \left(F_2(X) \times F_2(X)\right)$. Since $2^{E_0}\subseteq E$ we get that $(f_2 ^i(A),A_i) \in E\cap \left(F_2(X) \times F_2(X)\right)$ for each $i \geq N$.
\end{proof}

\begin{remark}
Example \ref{ExNotPreserveEventualInHyperspaces} shows that, in general, symmetric products do not preserve eventual shadowing for $n\geq3$.
\end{remark}

\subsection{Factor maps}

\begin{definition}
A factor map $\varphi\colon X\to Y$ is \emph{eALP} iff for every $E \in \mathscr{U}_Y$ and $D \in \mathscr{U}_X$ there is $V \in \mathscr{U}_Y$ such that for every $V$-pseudo-orbit $(y_i)_{i \in \omega}$ in $Y$ is a $D$-pseudo-orbit $(x_i)_{i \in \omega}$ in $X$ such that $(\varphi(x_i))_{i \in \omega}$ eventually-$E$-shadows $(y_i)_{i \in \omega}$.
\end{definition}

If $X$ and $Y$ are compact metric spaces, then $\varphi$ is eALP if and only if for all $\epsilon>0$ and $\eta>0$, there exists $\delta>0$ such that if $(y_i)_{i \in \omega}$ is a $\delta$-pseudo-orbit in $Y$, there exists an $\eta$-pseudo-orbit $(x_i)_{i \in \omega}$ in $X$ which eventually $\epsilon$-shadows $(y_i)_{i \in \omega}$.

\begin{theorem}\label{ThmALPeventualShad}
Suppose that $\varphi\colon (X,f)\to (Y,g)$ is a factor map.
\begin{enumerate}
	\item If $(X,f)$ has eventual shadowing and $\varphi$ eALP, then $(Y,g)$ has eventual shadowing.
	\item If $(Y,g)$ has eventual shadowing, then $\varphi$ eALP.
\end{enumerate}
\end{theorem}
\begin{proof}
 For (1), let $E \in \mathscr{U}_Y$ be given. Select $E_0 \in \mathscr{U}_Y$ with $2E_0 \subseteq E$. By the uniform continuity of $\varphi$ there exists $D_1 \in \mathscr{U}_X$ such that for all $a,b\in X$ with $(a,b)\in D_1$ one has $(\varphi(a),\varphi(b))\in E_0$. Next, let $D_2 \in \mathscr{U}_X$ be chosen so that $D_2$-pseudo-orbits in $X$ are $D_1$-eventually-
shadowed. Extract $W \in \mathscr{U}_Y$ from the definition of eALP using $E_0$ and $D_2$, we claim that $W$-pseudo-orbits of $(Y,g)$ are then eventually $E$-shadowed in $(Y,g)$. Indeed, given a $W$-pseudo-orbit $(y_i)_{i\in\omega}\subseteq Y$, there exists a $D_2$-pseudo-orbit $(x_i)_{i\in\omega}\subseteq X$ and $N \in \mathbb{N}$ such that, for all $i \geq N$, $(y_i,\varphi(x_i)) \in E_0$.
Consider $z\in X$ that $D_1$-eventually shadows $(x_i)_{i\in\omega}$. Let $M \in \mathbb{N}$ be such that  $(f^i(z),x_i) \in D_1$ for all $i\geq M$. Take $l=\max \{M,N\}$. Then, using uniform continuity and the triangle inequality, $(g^i(\varphi(z)), y_i) \in E$ for all $i \geq l$. Hence $\varphi(z)$ eventually $E$-shadows $(y_i)_{i\in\omega}$.

To see (2), fix $E \in \mathscr{U}_Y$ and $D \in \mathscr{U}_X$ and take $V\in \mathscr{U}$ to correspond to $E$ for eventual shadowing in $(Y,g)$. Let $(y_i)_{i\in\omega}$ be a $V$-pseudo-orbit in $(Y,g)$ and let $z\in Y$ eventually $E$-shadow it; let $N \in \mathbb{N}$ be such that $g^N(z)$ $E$-shadows $(y_{N+i})_{i \in \omega}$. Consider $x\in\varphi^{-1}(z)$ and define $x_i=f^i(x)$ for each $i\in\omega$ so that $(x_i)_{i\in\omega}$ is a $D$-pseudo-orbit in $(X,f)$. In particular, one then has that for all $i \geq N$\[	(\varphi(x_i),y_i)=(g^i(z),y_i)\in E.\]

\end{proof}

\subsection{Inverse limits}

\begin{theorem}
Let $(X,f)$ be conjugate to a Mittag-Leffler inverse limit system comprised of maps with eventual shadowing on compact Hausdorff spaces. Then $(X,f)$ has eventual shadowing.
\end{theorem}

\begin{proof}
Let $(\Lambda, \geq)$ be a directed set. For each $\lambda \in \Lambda$, let $(X_\lambda, f_\lambda)$ be a dynamical system on a compact Hausdorff space with eventual shadowing and let $((X_\lambda, f_\lambda), g^\eta _\lambda)$ be a Mittag-Leffler inverse system. Without loss of generality $(X,f)=(\varprojlim \{ X_\lambda, g_\lambda ^\eta\}, f )$.

Let $\mathcal{U}$ be a finite open cover of $X$. Since $X =\varprojlim \{ X_\lambda, g_\lambda ^\eta\}$ there exist $\lambda \in \Lambda$ and a finite open cover $\mathcal{W}_\lambda$ of $X_\lambda$ such that $\mathcal{W}\coloneqq \{\pi_\lambda ^{-1}(W) \cap X \mid W \in \mathcal{W}_\lambda \}$ refines $\mathcal{U}$. Now let $\gamma \in \Lambda$ witness the Mittag-Leffler condition with respect to $\lambda$. Let $\mathcal{W}_\gamma \coloneqq \{ g ^\gamma _ \lambda {}^{(-1)} (W) \colon W \in \mathcal{W}_\lambda\}$. By eventual shadowing for $(X_\gamma, f_\gamma)$ there exists a finite open cover $\mathcal{V}_\gamma$ of $X_\gamma$ such that every $\mathcal{V}_\gamma$-pseudo-orbit in $X_\gamma$ is eventually $\mathcal{W}_\gamma$-shadowed. Take $\mathcal{V}=\{ \pi_\gamma ^{-1}(V) \cap X \mid V \in \mathcal{V}_\gamma \}$ and suppose $(x_i)_{i \in \omega}$ is a $\mathcal{V}$-pseudo-orbit in $X$. It follows that $(\pi_\gamma(x_i))_{i \in \omega}$ is a $\mathcal{V}_\gamma$-pseudo-orbit in $X_\gamma$, which means there is a point $z \in X_\gamma$ which eventually $\mathcal{W}_\gamma$-shadows it. By construction, it follows that $g^\gamma _\lambda (z)$ eventually $\mathcal{W}_\lambda$-shadows $(\pi_\lambda(x_i))_{i \in \omega}$. Since the system is Mittag-Leffler there exists $y \in \pi_\lambda ^{-1}(g^\gamma _\lambda (z)) \cap X$. It follows that $y$ eventually $\mathcal{W}$-shadows $(x_i)_{i \in \omega}$. Since $\mathcal{W}$ is a refinement of $\mathcal{U}$ the result follows.
\end{proof}

\subsection{Tychonoff product}

\begin{theorem}
Let $\Lambda$ be an arbitrary index set and let $(X_\lambda, f_\lambda)$ be a compact Hausdorff system with eventual shadowing for each $\lambda \in \Lambda$. Then the product system $(X,f)$, where $X=\prod_{\lambda \in \Lambda} X_\lambda$, has eventual shadowing.
\end{theorem}

\begin{proof}
Let $E \in \mathscr{U}$ be given; this entourage is refined by one of the form
\[ \prod _{\lambda \in \Lambda} E_\lambda,\]
where $E_\lambda \in \mathscr{U}_\lambda$ for all $\lambda \in \Lambda$ and $E_\lambda =X_\lambda \times X_\lambda$ for all but finitely many of the $\lambda$'s. Let $\lambda_i$, for $1\leq i \leq k$, be precisely those elements in $\Lambda$ for which $E_\lambda  \neq X_\lambda \times X_\lambda$ (if there are no such elements then we are done). By eventual shadowing in each component space, there exist entourages $D_{\lambda_i} \in \mathscr{U}_{\lambda_i}$ such that every $D_{\lambda_i}$-pseudo-orbit is eventually $E_{\lambda_i}$-shadowed.
Let 
\[D \coloneqq \prod_{\lambda \in \Lambda} D_\lambda \]
where
\[D_\lambda = \left\{\begin{array}{lll}
X\times X  & \text{ if } & \forall i \, \lambda \neq \lambda_i
\\D_{\lambda_i} & \text{ if } & \exists i : \lambda=\lambda_i\,.
\end{array}\right.\]
Now let $(x_j)_{j \in \omega}$ be a $D$-pseudo-orbit. Then $(\pi_{\lambda_i}(x_j))_{j \in \omega}$ is a $D_{\lambda_i}$-pseudo-orbit in $X_{\lambda_i}$, which is eventually $E_{\lambda_i}$-shadowed by a point $z_i \in X_{\lambda_i}$; there exist $N_i$ such that $(\pi_{\lambda_i}(x_j))_{j \geq N_i}$ is $E_{\lambda_i}$-shadowed by $z_i$. Pick a point $z \in X$ such that $\pi_{\lambda_i}(z)= z_i$ for each $1\leq i \leq k$. Take $N= \max _{1\leq i \leq k} N_i$. Then $z$ $E$-shadows $(x_j)_{j \geq N}$. Thus, by definition, $z$ eventually $E$-shadows $(x_j)_{j \in \omega}$.
\end{proof}

\section{Preservation of Orbital Shadowing}

The orbital shadowing property was introduced in \cite{PiluginRodSakai2002} where the authors studied its relationship to classical stability properties, such as structural stability and 
$\Omega$-stability. It has since been studied by various other authors (e.g \cite{Pilyugin2007,GoodMeddaugh2016}).

Recall, a system $(X,f)$ has the \textit{orbital shadowing} property if for all $E \in \mathscr{U}$, there exists $D \in \mathscr{U}$ such that for any $D$-pseudo-orbit $( x_i)_{i \in \omega}$, there exists a point $z$ such that 

\[\left(\overline{\{x_i\}_{i\in\omega}}, \overline{\{f^i(z)\}_{i\in\omega}}\right) \in 2^E.\]

\subsection{Induced map on the hyperspace of compact sets}

\begin{theorem}\label{OrbHyp}
Let $X$ be a compact Hausdorff space, and let $f \colon X \to X$ be a continuous function. If the hyperspace system $(2^X, 2^f)$ witnesses orbital shadowing then the system $(X,f)$ experiences orbital shadowing.
\end{theorem}

\begin{proof}
Let $E \in \mathscr{U}$ be given and let $E_0 \in \mathscr{U}$ be such that $4E_0 \subseteq E$. Let $D \in \mathscr{U}$ be such that $2^D$ satisfies the condition for $2^{E_0}$ in orbital shadowing for the hyperspace. Let $( x_i ) _{i \in \omega}$ be a $D$-pseudo orbit in $X$. Then $( \{x_i\} ) _{i \in \omega}$ is a $2^D$-pseudo orbit in $2^X$. Then there exists $A \in 2^X$ such that
\[\left( \overline{\{\{x_i\}\}_{i\in\omega}} , \overline{\{(2^f)^i(A)\}_{i\in\omega}}\right) \in 2^{2^{E_0}}.\]
Equivalently
\begin{equation}\label{EquationOrb1}
    \overline{\{\{x_i\}\}_{i\in\omega}} \subseteq B_{2^{E_0}}\left(\overline{\{(2^f)^i(A)\}_{i\in\omega}}\right)
\end{equation}
and
\begin{equation}\label{EquationOrb2} \overline{\{(2^f)^i(A)\}_{i\in\omega}}\subseteq B_{2^{E_0}}\left(\overline{\{\{x_i\}\}_{i\in\omega}}\right).
\end{equation}
Pick $z \in A$. It can be verified that
\[\left(\overline{\{x_i\}_{i\in\omega}}, \overline{\{f^i(z)\}_{i\in\omega}}\right) \in 2^{4E_0}.\]
Indeed, suppose not. 

\textbf{Case i).} There exists $a \in \overline{\{x_i\}_{i\in\omega}}$ such that for any $b \in \overline{\{f^i(z)\}_{i\in\omega}}$ we have $(a,b) \notin 4E_0$. It follows that there exists $k \in \omega$ such that $(x_k, f^i(z))\notin 2E_0$ for all $i \in \omega$. We have from Equation (\ref{EquationOrb1}) that there exists $l \in \omega$ such that $\left(f^l(A), \{x_k\}\right) \in 2^{E_0}$; in particular, for any $y \in f^l(A)$, $(y, x_k)\in E_0$, a contradiction.

\textbf{Case ii).} There exists $b \in \overline{\{f^i(z)\}_{i\in\omega}}$ such that for any $a \in \overline{\{x_i\}_{i\in\omega}}$ we have $(b,a) \notin 4E_0$. It follows that there exists $k \in \omega$ such that $(f^k(z),x_i)\notin 2E_0$ for all $i \in \omega$. We have from Equation (\ref{EquationOrb2}) that there exists $l \in \omega$ such that $\left(f^k(A), \{x_l\}\right) \in 2^{E_0}$; in particular, for any $y \in f^k(A)$, $(y, x_l)\in E_0$, a contradiction.

It follows that 

\[\left(\overline{\{x_i\}_{i\in\omega}}, \overline{\{f^i(z)\}_{i\in\omega}}\right) \in 2^{4E_0} \subseteq 2^E.\]

\end{proof}
 
The following example shows that the converse to Theorem \ref{OrbHyp} is false.

\begin{example}\label{OrbHypExample}
Let $X$ be the circle $\mathbb{R}/\mathbb{Z}$ and let $f\colon X \to X$ be given by $x \mapsto x + \alpha$, where $\alpha$ is some fixed irrational number. Since $(X,f)$ is minimal it has strong orbital shadowing, and thereby orbital shadowing and first weak shadowing, by \cite[Corollary 2.7]
{Mitchell}. 
Let $x_0$ and $y_0$ be two antipodal points and let $\epsilon>0$ be given, with $\epsilon<\frac{1}{20}$. Suppose $\delta \in \mathbb{Q}$ with $0<\delta<\epsilon$. Then construct a $\delta$-pseudo orbit in $2^X$ recursively by the following rule: Let $A_0=\{x_0,y_0\}$ and, for all $i\in \omega\setminus\{0\}$, let $A_i=\{x_i,y_i\}:=\{f(x_{i-1})+\frac{\delta}{2}, f(y_{i-1}) +\frac{\delta}{3}\}$. We claim that this is not first weak shadowed. 
Suppose $A$ $\epsilon$-first-weak-shadows $( A_i ) _{i \in \omega}$; i.e.
\[B_{\epsilon}\left(\Orb(A)\right) \supseteq \{A_i\}_{i \in \omega}.\]
Then there exists $n \in \omega$ such that $d_H((2^f)^n(A), \{x_0, y_0\})< \epsilon$; thus $(2^f)^n(A) \subseteq B_\epsilon(x_0) \cup B_\epsilon(y_0)$, $(2^f)^n(A) \cap B_\epsilon(x_0) \neq \emptyset$ and $(2^f)^n(A) \cap B_\epsilon(y_0) \neq \emptyset$. Since $x_0$ and $y_0$ are antipodal and $f$ is an isometry, it follows that $A$ is a subset of a union of two antipodal arcs of length $\epsilon$ and that $A$ meets both these arcs; the same holds true of $(2^f)^i(A)$ for all $i \in \omega$. Now let $l \in \omega$ be least such that $d(x_l, y_l) \leq \frac{\delta}{6}$; such an $l$ exists by construction. 
We claim $\{x_l, y_l\} \notin B_\epsilon\left(\Orb(A)\right)$. Suppose not, then there exists $m \in \omega$ such that $d_H\left((2^f)^m(A), \{x_l, y_l\} \right) < \epsilon$. In particular, 
\[ (2^f)^m(A) \subseteq B_{\epsilon} \left(\{x_l, y_l\} \right). \]
But
\[B_{\epsilon} \left(\{x_l, y_l\} \right) \subseteq  B_{2\epsilon} \left(x_l \right),\]
and $B_{2\epsilon} \left(x_l \right)$ is an arc of length less than $\frac{4}{20}$ by construction, which does not contain any pair of antipodal points, contradicting our analysis of $(2^f)^i(A)$. Hence the hyperspatial system does not have first weak shadowing. Since 
\[\text{Strong orbital shadowing }\implies \text {orbital shadowing } \implies \text{first weak shadowing},\]
it also follows that the system has neither strong orbital shadowing nor orbital shadowing.
\end{example}

\subsection{Symmetric products}

The proof of Theorem \ref{OrbSymProd} is very similar to that of Theorem \ref{OrbHyp} and is thereby omitted.

\begin{theorem}\label{OrbSymProd}
Let $X$ be a compact Hausdorff space, and let $f \colon X \to X$ be a continuous function. For any $n \geq 2$, if the symmetric product system $(F_n(X), f_n)$ witnesses orbital shadowing then the system $(X,f)$ experiences orbital shadowing.
\end{theorem}
\begin{proof}
Omitted.
\end{proof}

\begin{remark}
The converse of Theorem \ref{OrbSymProd} is false. It is clear that Example \ref{OrbHypExample} may be suitably adjusted to provide a counterexample. Indeed, with sufficient adjustments, one can see that, for any $n \geq 2$, $(X,f)$ witnessing orbital shadowing does not generally imply that $(F_n(X),f_n)$ has orbital shadowing.
\end{remark}

\subsection{Factor maps}


\begin{definition}
Let $(X,f)$ and $(Y,g)$ be dynamical systems where $X$ and $Y$ are compact Hausdorff spaces. A factor map $\varphi \colon X\to Y$ is \emph{oALP} if for every $V \in \mathscr{U}_Y$ and $D \in \mathscr{U}_X$ there exists $W \in \mathscr{U}_Y$ such that for all $W$-pseudo-orbits $(y_i)\subseteq Y$, there exists a $D$-pseudo-orbit $(x_i)\subseteq X$ for which 
	\begin{align*}
    	(\varphi(\overline{\{x_i\}_{i\in\omega}}),\overline{\{y_i\}_{i\in\omega}}) \in 2^V.
     \end{align*}
\end{definition}

If $X$ and $Y$ are compact metric spaces, then $\varphi$ is oALP if and only if for all $\epsilon>0$ and $\eta>0$, there exists $\delta>0$ such that for all $\delta$-pseudo-orbits $(y_i)_{i \in \omega}$ in $Y$, there exists an $\eta$-pseudo-orbit $(x_i)_{i \in \omega}$ in $X$ such that the Hausdorff distance  $d_H(\varphi(\overline{\{x_i\}_{i\in\omega}}),\overline{\{y_i\}_{i\in\omega}})<\epsilon$.

\begin{theorem}\label{ThmALPOrbShad}
Suppose that $\varphi \colon (X,f)\to (Y,g)$ is a factor map.
	\begin{enumerate}
		\item If $(X,f)$ exhibits orbital shadowing and $\varphi$ is oALP, then $(Y,g)$ exhibits orbital shadowing.
		\item If $(Y,g)$ exhibits orbital shadowing, then $\varphi$ is oALP.
	\end{enumerate}
\end{theorem}
\begin{proof}
For (1), let $E \in \mathscr{U}_Y$ be given. Select $E_0 \in \mathscr{U}_Y$ with $2E_0 \subseteq E$. By the uniform continuity of $\varphi$ there exists $D_1 \in \mathscr{U}_X$ such that for all $a,b\in X$ with $(a,b)\in D_1$ one has $(\varphi(a),\varphi(b))\in E_0$. Next, let $D_2 \in \mathscr{U}_X$ be chosen so that $D_2$-pseudo-orbits in are $D_1$ orbital shadowed. Extract $W \in \mathscr{U}_Y$ from the definition of oALP using $E_0$ and $D_2$, we claim that $W$-pseudo-orbits of $(Y,g)$ are then $E$-orbital shadowed in $(Y,g)$. Indeed, given a $W$-pseudo-orbit $(y_i)_{i\in\omega}\subseteq Y$, there exists a $D_2$-pseudo-orbit $(x_i)_{i\in\omega}\subseteq X$ for which 
	\begin{align*}
		(\overline{\{y_i\}_{i\in\omega}},\varphi(\overline{\{x_i\}_{i\in\omega}}))\in 2^{E_0}.
	\end{align*}
Let $z\in X$ $D_1$-orbital shadow $(x_i)_{i\in\omega}$. Then, using uniform continuity and the triangle inequality, one may conclude that $\varphi(z)$ $E$-orbital shadows $(y_i)_{i\in\omega}$ as required.

For (2)fix $E \in \mathscr{U}_Y$ and $D \in \mathscr{U}_X$ and take $V\in \mathscr{U}_Y$ to correspond to $E$ for orbital shadowing in $(Y,g)$. Let $(y_i)_{i\in\omega}$ to be a $V$-pseudo-orbit in $(Y,g)$ and let $z\in Y$ $E$-orbital shadow it. Consider $x\in\varphi^{-1}(z)$ and define $x_i=f^i(x)$ for each $i\in\omega$ so that $(x_i)_{i\in\omega}$ is a $D$-pseudo-orbit in $(X,f)$. In particular, one then has that
	\begin{align*}
		(\varphi(\overline{\{x_i\}_{i\in\omega}}),\overline{\{y_i\}_{i\in\omega}})&=(\overline{\varphi(\{x_i\}_{i\in\omega})},\overline{\{y_i\}_{i\in\omega}})\\
		&=(\overline{\{g^i(z)\}_{i\in\omega}},\overline{\{y_i\}_{i\in\omega}})\in 2^E.
	\end{align*}
\end{proof}

\subsection{Inverse limits}
\begin{theorem} 
Let $(X,f)$ be conjugate to a Mittag-Leffler inverse limit system comprised of maps with orbital shadowing on compact Hausdorff spaces. Then $(X,f)$ has orbital shadowing. 
\end{theorem}

\begin{proof}
We use the reformulation of orbital shadowing given in Definition \ref{DefnOpenCoverStrongOrbShad}.

Let $(\Lambda, \geq)$ be a directed set. For each $\lambda \in \Lambda$, let $(X_\lambda, f_\lambda)$ be a dynamical system on a compact Hausdorff space with strong orbital shadowing and let $((X_\lambda, f_\lambda), g^\eta _\lambda)$ be a Mittag-Leffler inverse system. Without loss of generality $(X,f)=(\varprojlim \{ X_\lambda, g_\lambda ^\eta\}, f )$.

Let $\mathcal{U}$ be a finite open cover of $X$. Since $X =\varprojlim \{ X_\lambda, g_\lambda ^\eta\}$ there exist $\lambda \in \Lambda$ and a finite open cover $\mathcal{W}_\lambda$ of $X_\lambda$ such that $\mathcal{W}\coloneqq \{\pi_\lambda ^{-1}(W) \cap X \mid W \in \mathcal{W}_\lambda \}$ refines $\mathcal{U}$. Now let $\gamma \in \Lambda$ witness the Mittag-Leffler condition with respect to $\lambda$. Let $\mathcal{W}_\gamma \coloneqq \{ g ^\gamma _ \lambda {}^{(-1)} (W) \colon W \in \mathcal{W}_\lambda\}$. By orbital shadowing for $(X_\gamma, f_\gamma)$ there exists a finite open cover $\mathcal{V}_\gamma$ of $X_\gamma$ such that every $\mathcal{V}_\gamma$-pseudo-orbit in $X_\gamma$ is $\mathcal{W}_\gamma$-orbital-shadowed. Take $\mathcal{V}=\{ \pi_\gamma ^{-1}(V) \cap X \mid V \in \mathcal{V}_\gamma \}$ and suppose $(x_i)_{i \in \omega}$ is a $\mathcal{V}$-pseudo-orbit in $X$. It follows that $(\pi_\gamma(x_i))_{i \in \omega}$ is a $\mathcal{V}_\gamma$-pseudo-orbit in $X_\gamma$, which means there is a point $z \in X_\gamma$ which $\mathcal{W}_\gamma$-orbital-shadows it.
By construction, it follows that $g^\gamma _\lambda (z)$ $\mathcal{W}_\lambda$-orbital-shadows $(\pi_\lambda(x_i))_{i \in \omega}$, i.e.

\[ \forall y \in \overline{\Orb(g^\gamma _\lambda (z))} \exists W \in \mathcal{W}_\lambda \exists y^\prime \in \overline{\{\pi _\lambda \left(x_i\right)\}_{i \in \omega}} \colon y, y^\prime \in W,\]
and
\[ \forall y \in \overline{\{\pi _\lambda \left(x_i\right)\}_{i \in \omega}} \exists W \in \mathcal{W}_\lambda \exists y^\prime \in \overline{\Orb(g^\gamma _\lambda (z))} \colon y, y^\prime \in W.\]

Since the system is Mittag-Leffler there exists $z^\prime  \in \pi_\lambda ^{-1}(g^\gamma _\lambda (z)) \cap X$. We claim $z^\prime$ $\mathcal{U}$-orbital-shadows $(x_i)_{i \in \omega}$.


Let $y^\prime \in \overline{\Orb\left(z^\prime\right)}$. There exist $W \in W_\lambda$ and $k\in \omega$  
such that $\pi_\lambda(y^\prime), \pi_\lambda (x_k) \in W$. Then $a^\prime, x_k \in \pi _\gamma ^{-1}(W) \cap X \subseteq U$ for some $U \in \mathcal{U}$.

Now suppose $y^\prime \in \overline{\{ \left(x_i\right)\}_{i \in \omega}}$. There exist $W \in \mathcal{W}_\lambda$ and $k\in \omega$ such that $\pi_\lambda(y^\prime), f_{\lambda}^k(g^\gamma _\lambda (z)) \in W$. Then $y^\prime, f^k(z^\prime) \in \pi _\gamma ^{-1}(W) \cap X \subseteq U$ for some $U \in \mathcal{U}$.

\end{proof}

\subsection{Tychonoff product}

A product of systems with orbital shadowing does not necessarily have orbital shadowing. The following example demonstrates this.

\begin{example}\label{ExampleProdOrbNotOrb}
For $i \in \{1,2\}$ let $X_i= \mathbb{R}/\mathbb{Z}$, $d_i$ be the shortest arc length metric on $X_i$ and $f_i \colon X_i \to X_i \colon x \mapsto x + \alpha \mod 1$, where $\alpha$ is some fixed irrational number. Consider the product space $X=X_1 \times X_2$ with distance $d$ given by $d((a,b), (c,d))= \sup\{d_1(a, c), d_2(b,d)\}$. Recall that $(X_i,f_i)$ has strong orbital shadowing, and therefore shadowing, by \cite[Corollary 2.7]{Mitchell}.
It will be useful to define \[\lvert \cdot , \cdot \rvert \colon \mathbb{R}/\mathbb{Z} \times \mathbb{R}/\mathbb{Z} \to \mathbb{R} \colon (a,b) \mapsto \min \{\lvert a- b \rvert, \lvert b- a \rvert \}.\]

Now consider the product system $(X,f)$. Let $x_0 =0$ and $y_0 =\frac{1}{2}$ and let $\epsilon>0$ be given, with $\epsilon<\frac{1}{20}$. Suppose $\delta \in \mathbb{Q}$ with $0<\delta<\epsilon$. Then construct a $\delta$-pseudo orbit in $X$ recursively by the following rule: Let $z_0=(x_0 ,y_0)$ and, for all $i\in \omega\setminus\{0\}$, let $z_i=(f(x_{i-1})+\frac{\delta}{2}, f(y_{i-1}) +\frac{\delta}{3})$. We claim that this is not first weak shadowed. 
Suppose $z=(x,y)$ $\epsilon$-first-weak-shadows $( z_i ) _{i \in \omega}$; i.e.
\[B_\epsilon(\Orb(z)) \supseteq \{z_i\}_{i \in \omega}.\]
Then there exists $n \in \omega$ such that $d(f^n(z), (x_0 , y_0))< \epsilon$; that is, 
\[d((f_1 ^n(x), f_2 ^n(y)), (x_0 , y_0)) = \sup\{d_1(f_1 ^n(x), x_0), d_2(f^n _2 (y),y_0)\} < \epsilon.\]
In particular $d_1(f_1 ^n(x), x_0)< \epsilon$ and $d_2(f^n_2(y), y_0) <\epsilon$; hence $f^n(z) \in B_\epsilon(z_0)=B_\epsilon(x_0) \times B_\epsilon(y_0)$. 
It follows by the triangle inequality that $\lvert f_1 ^n(x), f_2 ^n(y) \rvert \geq \frac{1}{2} -2 \epsilon > \frac{3}{5}$.

Now let $l \in \omega$ be least such that $\lvert x_l, y_l \rvert \leq \frac{\delta}{6}$; such an $l$ exists by construction. 
We claim $(x_l , y_l) \notin B_\epsilon(\Orb(z))$. Suppose not, then there exists $m \in \omega$ such that $d\left(f^m(z), (x_l, y_l)\} \right) < \epsilon$; thus $d_1(f_1 ^m(x), x_l)< \epsilon$ and $d_2(f^m_2(y), y_l) <\epsilon$. 
It follows that $\lvert f^m _1(x), f^m _2(y) \rvert \leq  2\epsilon + \frac{\delta}{6}$. Since $f_1$ and $f_2$ are the same isometries it follows that $\lvert f_1 ^n(x), f_2 ^n(y) \rvert \leq  2\epsilon + \frac{\delta}{6} < \frac{1}{10}+\frac{1}{120} < \frac{1}{5}$. But we know $\lvert f_1 ^n(x), f_2 ^n(y) \rvert>\frac{3}{5}$, so we have a contradiction. It follows that $(x_l , y_l) \notin B_\epsilon(\Orb(z))$. Hence the product system does not have first weak shadowing (and thereby nor does it have orbital (resp. strong orbital) shadowing).
\end{example}

\section{Preservation of Strong Orbital Shadowing}
Strong orbital shadowing, a strengthening of orbital shadowing as the name suggests, was introduced in \cite{GoodMeddaugh2016} in the authors' pursuit of a characterisation of when the set of $\omega$-limit sets of a system coincides with the set of closed internally chain transitive sets.

The system $(X,f)$ has the \textit{strong orbital shadowing} property if for all $E \in \mathscr{U}$, there exists $D \in \mathscr{U}$ such that for any $D$-pseudo-orbit $( x_i)_{i \in \omega}$, there exists a point $z \in X$ such that, for all $N \in \omega$,

\[\left(\overline{\{x_{N+i}\}_{i\in\omega}}, \overline{\{f^{N+i}(z)\}_{i\in\omega}}\right) \in 2^E.\]

\subsection{Induced map on the hyperspace of compact sets}

\begin{theorem}\label{StrongOrbHyp}
Let $X$ be a compact Hausdorff space, and let $f \colon X \to X$ be a continuous function. If the hyperspace system $(2^X, 2^f)$ has strong orbital shadowing then the system $(X,f)$ has strong orbital shadowing.
\end{theorem}

\begin{proof}
Let $E \in \mathscr{U}$ be given and let $E_0 \in \mathscr{U}$ be such that $4E_0 \subseteq E$. Let $D \in \mathscr{U}$ be such that $2^D$ satisfies the the condition for $2^{E_0}$ in orbital shadowing for the hyperspace. Let $( x_i ) _{i \in \omega}$ be a $D$-pseudo orbit in $X$. Then $( \{x_i\} ) _{i \in \omega}$ is a $2^D$-pseudo orbit in $2^X$ and there exists $A \in 2^X$ such that for any $N \in \omega$, $f^N(A)$ we have
\[\left( \overline{\{\{x_{N+i}\}\}_{i\in\omega}} , \overline{\{(2^f)^{N+i}(A)\}_{i\in\omega}}\right) \in 2^{2^{E_0}}.\]
Equivalently, for any $N \in \omega$
\begin{equation}\label{EquationStrongOrb1}
    \overline{\{\{x_{N+i}\}\}_{i\in\omega}} \subseteq B_{2^{E_0}}\left(\overline{\{(2^f)^{N+i}(A)\}_{i\in\omega}}\right)
\end{equation}
and
\begin{equation}\label{EquationStrongOrb2} \overline{\{(2^f)^{N+i}(A)\}_{i\in\omega}}\subseteq B_{2^{E_0}}\left(\overline{\{\{x_{N+i}\}\}_{i\in\omega}}\right).
\end{equation}
Pick $z \in A$. It can be verified that for any $N \in \omega$
\[\left(\overline{\{x_{N+i}\}_{i\in\omega}}, \overline{\{f^{N+i}(z)\}_{i\in\omega}}\right) \in 2^{4E_0}.\]
Indeed, suppose not. 

\textbf{Case i).} There exist $N \in \omega$ and $a \in \overline{\{x_{N+i}\}_{i\in\omega}}$ such that for any $b \in \overline{\{f^{N+i}(z)\}_{i\in\omega}}$ we have $(a,b) \notin 4E_0$. It follows that there exists $k \in \omega$ such that $(x_{N+k}, f^{N+i}(z))\notin 2E_0$ for all $i \in \omega$. We have from Equation (\ref{EquationStrongOrb1}) that there exists $l \in \omega$ such that $\left(f^{N+l}(A), \{x_{N+k}\}\right) \in 2^{E_0}$; in particular, for any $y \in f^{N+l}(A)$, $(y, x_k)\in E_0$, a contradiction.

\textbf{Case ii).} There exist $N \in \omega$ and $b \in \overline{\{f^{N+i}(z)\}_{i\in\omega}}$ such that for any $a \in \overline{\{x_{N+i}\}_{i\in\omega}}$ we have $(b,a) \notin 4E_0$. It follows that there exists $k \in \omega$ such that $(f^{N+k}(z),x_{N+i})\notin 2E_0$ for all $i \in \omega$. We have from Equation (\ref{EquationStrongOrb2}) that there exists $l \in \omega$ such that $\left(f^{N+k}(A), \{x_{N+l}\}\right) \in 2^{E_0}$; in particular, for any $y \in f^{N+k}(A)$, $(y, x_{N+l})\in E_0$, a contradiction.

It follows that for any $N \in \omega$

\[\left(\overline{\{x_{N+i}\}_{i\in\omega}}, \overline{\{f^{N+i}(z)\}_{i\in\omega}}\right) \in 2^{4E_0} \subseteq 2^E.\]
\end{proof}

\begin{remark}
Example \ref{OrbHypExample} shows that the converse to Theorem \ref{StrongOrbHyp} is not true. The hyperspatial system does not have orbital shadowing, therefore nor does it have strong orbital shadowing.
\end{remark}

\subsection{Symmetric products}

The proof of Theorem \ref{StrongOrbSymProd} is very similar to that of Theorem \ref{StrongOrbHyp} and is thereby omitted.

\begin{theorem}\label{StrongOrbSymProd}
Let $X$ be a compact Hausdorff space, and let $f \colon X \to X$ be a continuous function. For any $n \geq 2$, if the symmetric product system $(F_n(X), f_n)$ witnesses orbital shadowing then system $(X,f)$ experiences orbital shadowing.
\end{theorem}
\begin{proof}
Omitted.
\end{proof}

\begin{remark}
The converse of Theorem \ref{StrongOrbSymProd} is false. It is clear that Example \ref{OrbHypExample} may be suitably adjusted to provide a counterexample. Indeed, with sufficient adjustments, one can see that, for any $n\geq 2$, $(X,f)$ witnessing strong orbital shadowing does not generally imply that $(F_n(X),f_n)$ has strong orbital shadowing.
\end{remark}

\subsection{Factor maps}


\begin{definition}
Let $(X,f)$ and $(Y,g)$ be dynamical systems where $X$ and $Y$ are compact Hausdorff spaces. A surjective semiconjugacy $\varphi\colon X\to Y$ is \emph{soALP} if  for every $E \in \mathscr{U}_Y$ and $D \in \mathscr{U}_X$ there exists $W \in \mathscr{U}_Y$ such that for all $W$-pseudo-orbits $(y_i)_{i \in \omega}\subseteq Y$, there exists a $D$-pseudo-orbit $(x_i)_{i \in \omega}\subseteq X$ such that for all $N \in \omega$,
	\begin{align*}
    	(\varphi(\overline{\{x_{N+i}\}_{i\in\omega}}),\overline{\{y_{N+i}\}_{i\in\omega}}) \in 2^E.
     \end{align*}
\end{definition}

If $X$ and $Y$ are compact metric spaces, then $\varphi$ is soALP if and only if for all $\epsilon>0$ and $\eta>0$, there exists $\delta>0$ such that for all $\delta$-pseudo-orbits $(y_i)_{i \in \omega}$ in $Y$, there exists an $\eta$-pseudo-orbit $(x_i)_{i \in \omega}$ in $X$
such that for all $N \in \omega$, the Hausdorff distance  $d_H(\varphi(\overline{\{x_{N+i}\}_{i\in\omega}}),\overline{\{y_{N+i}\}_{i\in\omega}})<\epsilon$.

\begin{theorem}\label{ThmALPStrongOrbShad}
Suppose that $\varphi\colon (X,f)\to (Y,g)$ is a surjective semiconjugacy.
	\begin{enumerate}
		\item If $(X,f)$ exhibits strong orbital shadowing and $\varphi$ is soALP, then $(Y,g)$ exhibits strong orbital shadowing.
		\item If $(Y,g)$ exhibits strong orbital shadowing, then $\varphi$ is soALP.
	\end{enumerate}
\end{theorem}
\begin{proof}
For (1), let $E \in \mathscr{U}_Y$ be given. Select $E_0 \in \mathscr{U}_Y$ with $2E_0 \subseteq E$. By the uniform continuity of $\varphi$ there exists $D_1 \in \mathscr{U}_X$ such that for all $a,b\in X$ with $(a,b)\in D_1$ one has $(\varphi(a),\varphi(b))\in E_0$. Next, let $D_2 \in \mathscr{U}_X$ be chosen so that $D_2$-pseudo-orbits in $X$ are $D_1$ strong orbital shadowed. Using $E_0$ and $D_2$ in the definition of soALP then provides $W \in \mathscr{U}_Y$ and we claim that $W$-pseudo-orbits in $(Y,g)$ are $E$-strong-orbital-shadowed. Indeed, given a $W$-pseudo-orbit $(y_i)_{i\in\omega}\subseteq Y$, there exists a $D_2$-pseudo-orbit $(x_i)_{i\in\omega}\subseteq X$ such that, for all $N \in \omega$, 
	\begin{align*}
		(\overline{\{y_{N+i}\}_{i\in\omega}},\varphi(\overline{\{x_{N+i}\}_{i\in\omega}}))\in 2^{E_0}.
	\end{align*}
Let $z\in X$ $D_1$-strong-orbital shadows $(x_i)_{i\in\omega}$. Then, using uniform continuity and the triangle inequality, one may conclude that $\varphi(z)$ $E$-strong-orbital shadows $(y_i)_{i\in\omega}$ as required.

For (2), fix $E \in \mathscr{U}_Y$ and $D \in \mathscr{U}_X$ and take $V\in \mathscr{U}$ to correspond to $E$ for strong orbital shadowing in $(Y,g)$. Let $(y_i)_{i\in\omega}$ to be an $V$-pseudo-orbit in $(Y,g)$ and let $z\in Y$ $E$-strong-orbital shadow it. Consider $x\in\varphi^{-1}(z)$ and define $x_i=f^i(x)$ for each $i\in\omega$ so that $(x_i)_{i\in\omega}$ is a $D$-pseudo-orbit in $(X,f)$. In particular, for any $N \in \omega$, one then has that
	\begin{align*}
		(\varphi(\overline{\{x_{N+i}\}_{i\in\omega}}),\overline{\{y_{N+i}\}_{i\in\omega}})&=(\overline{\varphi(\{x_{N+i}\}_{i\in\omega})},\overline{\{y_{N+i}\}_{i\in\omega}})\\
		&=(\overline{\{g^{N+i}(z)\}_{i\in\omega}},\overline{\{y_{N+i}\}_{i\in\omega}})\in 2^E.
	\end{align*}
\end{proof}

\subsection{Inverse limits}
\begin{theorem} 
Let $(X,f)$ be conjugate to a Mittag-Leffler inverse limit system comprised of maps with strong orbital shadowing on compact Hausdorff spaces. Then $(X,f)$ has strong orbital shadowing. 
\end{theorem}

\begin{proof}
Let $(\Lambda, \geq)$ be a directed set. For each $\lambda \in \Lambda$, let $(X_\lambda, f_\lambda)$ be a dynamical system on a compact Hausdorff space with strong orbital shadowing and let $((X_\lambda, f_\lambda), g^\eta _\lambda)$ be a Mittag-Leffler inverse system. Without loss of generality $(X,f)=(\varprojlim \{ X_\lambda, g_\lambda ^\eta\}, f )$.

Let $\mathcal{U}$ be a finite open cover of $X$. Since $X =\varprojlim \{ X_\lambda, g_\lambda ^\eta\}$ there exist $\lambda \in \Lambda$ and a finite open cover $\mathcal{W}_\lambda$ of $X_\lambda$ such that $\mathcal{W}\coloneqq \{\pi_\lambda ^{-1}(W) \cap X \mid W \in \mathcal{W}_\lambda \}$ refines $\mathcal{U}$. Now let $\gamma \in \Lambda$ witness the Mittag-Leffler condition with respect to $\lambda$. Let $\mathcal{W}_\gamma \coloneqq \{ g ^\gamma _ \lambda {}^{(-1)} (W) \colon W \in \mathcal{W}_\lambda\}$. By strong orbital shadowing for $(X_\gamma, f_\gamma)$ there exists a finite open cover $\mathcal{V}_\gamma$ of $X_\gamma$ such that every $\mathcal{V}_\gamma$-pseudo-orbit in $X_\gamma$ is $\mathcal{W}_\gamma$-strong-orbital-shadowed. Take $\mathcal{V}=\{ \pi_\gamma ^{-1}(V) \cap X \mid V \in \mathcal{V}_\gamma \}$ and suppose $(x_i)_{i \in \omega}$ is a $\mathcal{V}$-pseudo-orbit in $X$. It follows that $(\pi_\gamma(x_i))_{i \in \omega}$ is a $\mathcal{V}_\gamma$-pseudo-orbit in $X_\gamma$, which means there is a point $z \in X_\gamma$ which $\mathcal{W}_\gamma$-strong-orbital-shadows it. By construction, it follows that $g^\gamma _\lambda (z)$ $\mathcal{W}_\lambda$-strong-orbital-shadows $(\pi_\lambda(x_i))_{i \in \omega}$. Since the system is Mittag-Leffler there exists $y \in \pi_\lambda ^{-1}(g^\gamma _\lambda (z)) \cap X$. It follows that $y$ $\mathcal{W}$-strong-orbital-shadows $(x_i)_{i \in \omega}$. Since $\mathcal{W}$ is a refinement of $\mathcal{U}$ the result follows.
\end{proof}

\subsection{Tychonoff product}

\begin{remark}
A product of systems with strong orbital shadowing need not have strong orbital shadowing. The component systems in Example \ref{ExampleProdOrbNotOrb} have strong orbital shadowing however their product does not have strong orbital shadowing (since it does not have orbital shadowing).
\end{remark}
\section{First Weak Shadowing}

First weak shadowing was introduced by the authors in \cite{CorlessPilyugin} where it was called weak shadowing. The name was revised to first weak shadowing in \cite{PiluginRodSakai2002} to accommodate another similar natural weakening of shadowing, i.e. second weak shadowing.

As stated in Section \ref{ShadowingTypes}, a dynamical system $(X,f)$ has the \textit{first weak shadowing} property if for all $E \in \mathscr{U}$, there exists $D \in \mathscr{U}$ such that for any $D$-pseudo-orbit $( x_i)_{i \in \omega}$, there exists a point $z$ such that

\[\{x_i\}_{i\in\omega} \subseteq B_E\left( \Orb(z)\right).\]

\subsection{Induced map on the hyperspace of compact sets}

\begin{theorem}\label{FirstWeakHyp}
Let $X$ be a compact Hausdorff space, and let $f \colon X \to X$ be a continuous function. If the hyperspace system $(2^X, 2^f)$ witnesses the first weak shadowing property then system $(X,f)$ has first weak shadowing.
\end{theorem}

\begin{proof}
Let $E \in \mathscr{U}$ be given. Let $D \in \mathscr{U}$ be such that $2^D$ corresponds to $2^E$ in first weak shadowing for $2^f$. Let $(x_i)_{i \in \omega}$ be a $D$-pseudo-orbit in $X$. We then have that $\left(\{x_i\}\right)_{i\in \omega}$ is a $2^D$-pseudo-orbit in $2^X$; let $A \in 2^X$ be such that 
\[\left\{\{x_i\}\right\}_{i\in \omega} \subseteq B_{2^E}\left(\Orb(A)\right).\]
Fix $a \in A$. Pick $l \in \omega$ arbitrarily. There exists $m \in \omega$ such that $\{x_l\} \in B_{2^E}\left(f^m(A)\right)$, i.e. $\left(\{x_l\}, f^m(A)\right) \in 2^E$. This implies $(x_l,f^m(a)) \in E$. Since $l \in \omega$ was picked arbitrarily it follows that 
\[\{x_i\}_{i\in \omega} \subseteq B_E\left(\Orb(a)\right).\]

\end{proof}

\begin{remark}
Example \ref{OrbHypExample} shows that the converse to Theorem \ref{FirstWeakHyp} is not true; in general hyperspace systems do not preserve the first weak shadowing property.
\end{remark}

\subsection{Symmetric products}

The proof of Theorem \ref{FirstWeakSymProd} is very similar to that of Theorem \ref{FirstWeakHyp} and is thereby omitted.

\begin{theorem}\label{FirstWeakSymProd}
Let $X$ be a compact Hausdorff space, and let $f \colon X \to X$ be a continuous onto function. For any $n \geq 2$, if the symmetric product system $(F_n(X), f_n)$ has first weak shadowing then $(X,f)$ has first weak shadowing.
\end{theorem}
\begin{proof}
Omitted.
\end{proof}

\begin{remark}
The converse of Theorem \ref{FirstWeakSymProd} is false. It is clear that Example \ref{OrbHypExample} may be suitably adjusted to provide a counterexample. Indeed, with sufficient adjustments, one can see that, for any $n \geq 2$, $(X,f)$ witnessing first weak shadowing does not generally imply that $(F_n(X),f_n)$ has first weak shadowing.
\end{remark}

\subsection{Factor maps}


\begin{definition}
Let $(X,f)$ and $(Y,g)$ be dynamical systems where $X$ and $Y$ are compact Hausdorff spaces. A factor map $\varphi\colon X\to Y$ between compact Hausdorff spaces $X$ and $Y$ is \emph{w1ALP} if for every $V \in \mathscr{U}_Y$ and $D \in \mathscr{U}_X$ there exists $W \in \mathscr{U}_Y$ such that for all $W$-pseudo-orbits $(y_i)\subseteq Y$, there exists a $D$-pseudo-orbit $(x_i)\subseteq X$ for which $\{y_i\}_{i \in \omega} \subseteq B_{E}(\{\varphi(x_i)\}_{i \in \omega})$.
\end{definition}

If $X$ and $Y$ are compact metric spaces, then $\varphi$ is w1ALP if and only if for every $\epsilon>0$ and $\eta>0$ there exists $\delta>0$ such that for any $\delta$-pseudo-orbit in $Y$, there exists an $\eta$-pseudo-orbit in $X$ for which 
$\{y_i\}_{i \in \omega} \subseteq B_{\epsilon}(\{\varphi(x_i)\}_{i \in \omega})$.

\begin{theorem}\label{ThmALPFirstWeakShad}
Suppose that $\varphi \colon (X,f)\to (Y,g)$ is a factor map.
	\begin{enumerate}
		\item If $(X,f)$ exhibits first weak shadowing and $\varphi$ is w1ALP, then $(Y,g)$ exhibits first weak shadowing.
		\item If $(Y,g)$ exhibits first weak shadowing, then $\varphi$ is w1ALP.
	\end{enumerate}
\end{theorem}
\begin{proof}
For (1), let $E \in \mathscr{U}_Y$ be given. Select $E_0 \in \mathscr{U}_Y$ with $2E_0 \subseteq E$. By the uniform continuity of $\varphi$ there exists $D_1 \in \mathscr{U}_X$ such that for all $a,b\in X$ with $(a,b)\in D_1$ one has $(\varphi(a),\varphi(b))\in E_0$. Next, let $D_2 \in \mathscr{U}_X$ be chosen so that $D_2$-pseudo-orbits in $X$ are $D_1$ first weak shadowed. Extract $W \in \mathscr{U}_Y$ from the definition of w1ALP using $E_0$ and $D_2$, we claim that $W$-pseudo-orbits of $(Y,g)$ are then $E$ first weak shadowed in $(Y,g)$. Indeed, let $(y_i)_{i\in\omega}\subseteq Y$ be a $W$-pseudo-orbit and let $(x_i)_{i\in\omega}\subseteq X$ be a $D_2$-pseudo-orbit lifted through $\varphi$, that is,
	\begin{align*}
		\{y_i\}_{i\in\omega}\subseteq B_{E_0}(\{x_i\}_{i\in\omega}).
	\end{align*}
Suppose $z\in X$ $D_1$ first weak shadows $(x_i)_{i\in\omega}$ so that for each $i\in\omega$, there exists $j\in\omega$ such that $(x_i,f^j(z)) \in D_1$. Then
	\begin{align*}
		(\varphi(x_i),\varphi(f^j(z)))=(\varphi(x_i),g^j(\varphi(z))) \in E_0.
	\end{align*}
In turn, this provides
	\begin{align*}
		\{\varphi(x_i)\}_{i\in\omega}\subseteq B_{E_0}(\Orb(\varphi(z))),
	\end{align*}
and hence,
	\begin{align*}
		\{y_i\}_{i\in\omega}\subseteq B_{E_0}(\{x_i\}_{i\in\omega})\subseteq B_{E}(\Orb(\varphi(z))).
	\end{align*}

For (2), let $E \in \mathscr{U}_Y$ and $D \in \mathscr{U}_X$ be given and let $V \in \mathscr{U}_Y$ exhibit $E$ first weak shadowing in $(Y,g)$. Consider a $V$-pseudo-orbit $(y_i)_{i\in\omega}\subseteq Y$ and let $z\in Y$ $E$ first weak shadow it. By surjectivity of $\varphi$, there exists $x\in\varphi^{-1}(z)$ so one may construct the orbit $(x_i)_{i\in\omega}=(f^i(x))_{i\in\omega}$ which is trivially a $D$-pseudo-orbit. Then, for all $i\in\omega$ there exists $j\in\omega$ such that 
	\begin{align*}
		(y_i,g^j(z))=(y_i,\varphi(x_j))\in E,
	\end{align*}
and hence,
	\begin{align*}
		\{y_i\}_{i\in\omega}\subseteq B_E(\{\varphi(x_i)\}_{i\in\omega})
	\end{align*}
so that $\varphi$ is w1ALP.
\end{proof}

\subsection{Inverse limits}

\begin{theorem} 
Let $(X,f)$ be conjugate to a Mittag-Leffler inverse limit system comprised of maps with first weak shadowing on compact Hausdorff spaces. Then $(X,f)$ has first weak shadowing. 
\end{theorem}

\begin{proof}
Let $(\Lambda, \geq)$ be a directed set. For each $\lambda \in \Lambda$, let $(X_\lambda, f_\lambda)$ be a dynamical system on a compact Hausdorff space with first weak shadowing and let $((X_\lambda, f_\lambda), g^\eta _\lambda)$ be a Mittag-Leffler inverse system. Without loss of generality $(X,f)=(\varprojlim \{ X_\lambda, g_\lambda ^\eta\}, f )$.

Let $\mathcal{U}$ be a finite open cover of $X$. Since $X =\varprojlim \{ X_\lambda, g_\lambda ^\eta\}$ there exist $\lambda \in \Lambda$ and a finite open cover $\mathcal{W}_\lambda$ of $X_\lambda$ such that $\mathcal{W}\coloneqq \{\pi_\lambda ^{-1}(W) \cap X \mid W \in \mathcal{W}_\lambda \}$ refines $\mathcal{U}$. Now let $\gamma \in \Lambda$ witness the Mittag-Leffler condition with respect to $\lambda$. Let $\mathcal{W}_\gamma \coloneqq \{ g ^\gamma _ \lambda {}^{(-1)} (W) \colon W \in \mathcal{W}_\lambda\}$. By first weak shadowing for $(X_\gamma, f_\gamma)$ there exists a finite open cover $\mathcal{V}_\gamma$ of $X_\gamma$ such that every $\mathcal{V}_\gamma$-pseudo-orbit in $X_\gamma$ is $\mathcal{W}_\gamma$-first-weak-shadowed. Take $\mathcal{V}=\{ \pi_\gamma ^{-1}(V) \cap X \mid V \in \mathcal{V}_\gamma \}$ and suppose $(x_i)_{i \in \omega}$ is a $\mathcal{V}$-pseudo-orbit in $X$. It follows that $(\pi_\gamma(x_i))_{i \in \omega}$ is a $\mathcal{V}_\gamma$-pseudo-orbit in $X_\gamma$, which means there is a point $z \in X_\gamma$ which $\mathcal{W}_\gamma$-first-weak-shadows it. By construction, it follows that $g^\gamma _\lambda (z)$ $\mathcal{W}_\lambda$-first-weak-shadows $(\pi_\lambda(x_i))_{i \in \omega}$. Since the system is Mittag-Leffler there exists $y \in \pi_\lambda ^{-1}(g^\gamma _\lambda (z)) \cap X$. It follows that $y$ $\mathcal{W}$-first-weak-shadows $(x_i)_{i \in \omega}$. Since $\mathcal{W}$ is a refinement of $\mathcal{U}$ the result follows.
\end{proof}

\subsection{Tychonoff product}

\begin{remark}
A product of systems with first weak shadowing need not have first weak shadowing. Example \ref{ExampleProdOrbNotOrb} demonstrates this.
\end{remark}
\section{Preservation of Second Weak Shadowing}

The compact metric version of second weak shadowing was first introduced in \cite{PiluginRodSakai2002}. 
Recall that a system $(X,f)$ has the second weak shadowing property if for all $E \in \mathscr{U}$, there exists $D \in \mathscr{U}$ such that for any $D$-pseudo-orbit $( x_i)_{i \in \omega}$, there exists a point $z$ such that

\[\Orb(z) \subseteq B_E\left( \{x_i\}_{i\in\omega}\right).\]

Pilyugin \textit{et al} \cite{PiluginRodSakai2002} show that every compact metric system exhibits this property. This result extends to a compact Hausdorff setting \cite{Mitchell}. Since the hyperspace, symmetric product, inverse limit and tychonoff product of (a) compact Hausdorff system(s) are themselves compact Hausdorff it follows that any of these induced systems will also have the second weak shadowing property.
\section{Preservation of Limit Shadowing}

Limit shadowing was introduced in \cite{EirolarNevanlinnaPilyugin} with reference to hyperbolic sets. Recall that a system $(X, f)$ is said to have \textit{limit shadowing} if every asymptotic pseudo-orbit is asymptotically shadowed.

\subsection{Induced map on the hyperspace of compact sets}

\begin{theorem}\label{ThmLimShadHyp} Let $X$ be a compact Hausdorff space and let $f \colon X \to X$ be a continuous function. If $(2^X,2^f)$ has limit shadowing then $(X,f)$ has limit shadowing.
\end{theorem}

\begin{proof}
Let $(x_i)_{i \in \omega}$ be an asymptotic pseudo-orbit in $X$. Then $(\{x_i\})_{i \in \omega}$ is an asymptotic pseudo-orbit in $2^X$; this is asymptotically shadowed by a set $A \in 2^X$. Pick $a \in A$. It is easy to verify that $a$ asymptotically shadows $(x_i)_{i \in \omega}$.


\end{proof}

\subsection{Symmetric products}

The proof of Theorem \ref{ThmLimShadSymProd} is very similar to that of Theorem \ref{ThmLimShadHyp} and is thereby omitted.

\begin{theorem}\label{ThmLimShadSymProd}
Let $X$ be a compact Hausdorff space, and let $f \colon X \to X$ be a continuous onto function. For any $n \geq 2$, if the symmetric product system $(F_n(X), f_n)$ has limit shadowing then $(X,f)$ has limit shadowing.
\end{theorem}
\begin{proof}
Omitted.
\end{proof}

\begin{theorem}
Let $X$ be a compact Hausdorff space and let $f \colon X \to X$ be a continuous function. If $(X,f)$ has limit shadowing then $(F_2(X),f_2)$ has limit shadowing.
\end{theorem}
\begin{proof}
Suppose that $(A_i)_{i \in \omega}$ is an asymptotic pseudo-orbit in $F_2(X)$. Write $A_i=\{x_i, y_i\}$; it is possible that, for some $i$, $x_i=y_i$. We may relabel the $x$'s and $y$'s where necessary to give asymptotic pseudo-orbits $(x_i)_{i\in \omega}$ and $(y_i)_{i \in \omega}$ in $X$. 
By limit shadowing there exist $x,y \in X$ which asymptotically shadow $(x_i)_{i \in \omega}$ and $(y_i)_{i \in \omega}$ respectively. Write $A=\{x,y\} \in F_2(X)$. It is now straightforward to verify that $A$ asymptotically shadows $(A_i)_{i \in \omega}$.
\end{proof}

\begin{example}\label{ExTentMapDoesNotPreserveLimitShadToSymProducts}
Let $X$ be the closed unit interval and let $f \colon X \to X$ be the standard tent map, i.e.
 \[f(x)=\left\{\begin{array}{lll}
2x & \text{ if } & x \in [0,\frac{1}{2}]
\\2(1-x) & \text{ if } & t \in (\frac{1}{2},1] \,.
\end{array}\right.\]
Then $f$ has s-limit shadowing and limit shadowing  (see \textup{\cite{BarwellGoodOprocha}}) but $f_n$ does not have limit shadowing (and consequently it does not have s-limit shadowing) for any $n \geq 3$.

Fix $n \geq 3$. Let $c= \frac{2}{3}$. Let $\epsilon = \frac{1}{12}$ and let $\delta>0$ be given; without loss of generality $\delta<\frac{1}{12}$. Choose $y\in [0,\delta)$ such that there exists $k \in \mathbb{N}$ such that $f^k(y) =c$ and $f^i(y)< \frac{1}{2}$ for all $i <k$ and set $y_0:=y$ and let $y_l = \frac{y_{l-1}}{2}$ for all $l \in \mathbb{N}$. Note that $f^{k+l}(y_l)=c$ for any $l \in \omega$ and $f^{l+i}(y_l)<c$ for all $i <k$.

Construct an asymptotic $\delta$-pseudo-orbit as follows. Let $A_0=\{0,y_0,c\}$, $A_1=\{0, f(y_0),c\}$, $A_2=\{0, f^2(y_0),c \}$... $A_k=\{0, c\}$, $A_{k+1}=\{0, y_1, c\}$, $A_{k+2}=\{0, f(y_1),c\}$, $A_{k+3}=\{0, f^2(y_1),c \}$... $A_{2k+l+1}=\{0, c\}$, $A_{2k+l+2}=\{0, y_2, c\}$....
Explicitly, $A_{mk+(m-1)l+m}=\{0, y_m, c\}$ and $A_{mk+(m-1)l+m +i}=\{0, f^i(y_m), c\}$ for all $m\in \mathbb{N}$ and $1\leq i < k+m$. It is easy to see that $(A_i)_{i \in \omega}$ is an asymptotic $\delta$-pseudo-orbit. Suppose that $A \in F_n(X)$ asymptotically shadows this asymptotic $\delta$-pseudo-orbit. It follows that it eventually $\frac{1}{12}$-shadows $(A_i)_{i \in \omega}$; there exists $N \in \mathbb{N}$ such that $f_n ^N(A)$ $\frac{1}{12}$-shadows $(A_{N+i})_{i \in \omega}$


First observe that, since the pseudo-orbit is always a subset of the interval $[0,\frac{2}{3}]$, shadowing entails that $f_n ^{N+i}(A) \subseteq [0,\frac{3}{4})$ for any $i \in \omega$. Finally, every point in $f_n ^N(A)$ must either be $0$, $c$ or a preimage of $c$ in the interval $[0, \frac{2}{3}]$, otherwise it would enter (or already lie in) $[\frac{3}{4}, 1]$ which would contradict shadowing. Now let $z$ be the least such element of $f_n ^N(A) \setminus \{0\}$. Let $p\in \omega$ be least such that $f^p(z)=c$. Then $f_n ^{N+p+i}(A)=\{0,c\}$ for all $i \in \omega$; clearly this contradicts the fact that $f_n ^N(A)$ is $\frac{1}{12}$-shadowing $(A_{N+i})_{{N+i} \in \omega}$ since there exists $q > N+p+i$ such that $A_{q}=\{0,c, \frac{1}{3}\}$. Therefore $f_n$ does not have limit shadowing (resp. s-limit shadowing).
\end{example}

\subsection{Factor maps}

\begin{definition} Suppose $X$ and $Y$ are compact Hausdorff spaces and $f\colon X \to X$, $g \colon Y \to Y$ are continuous. A surjective semiconjugacy
$\varphi\colon X\to Y$ is \emph{ALAP} iff for every asymptotic pseudo-orbit $(y_i)_{i \in \omega}\subseteq Y$, there exists an asymptotic-pseudo-orbit $(x_i)_{i \in \omega}\subseteq X$ such that $\varphi(x_i)$ asymptotically shadows $(y_i)$.
\end{definition}

The proof of the following is similar to the proofs of Theorems \ref{ThmALPeventualShad}, \ref{ThmALPOrbShad}, \ref{ThmALPStrongOrbShad} and \ref{ThmALPFirstWeakShad} and is therefore omitted.

\begin{theorem}\label{ThmALPLimShad}
Suppose that $\varphi \colon (X,f)\to (Y,g)$ is a surjective semiconjugacy.
	\begin{enumerate}
		\item If $(X,f)$ exhibits limit shadowing and $\varphi$ is \emph{ALAP}, then $(Y,g)$ exhibits limit shadowing.
		\item If $(Y,g)$ exhibits limit shadowing, then $\varphi$ is \emph{ALAP}.
	\end{enumerate}
\end{theorem}

\subsection{Inverse limits}
Whilst it remains unclear whether general inverse limit systems preserve limit shadowing, we note the following result proved by the first author \textit{et al} in \cite{GoodOprochaPuljiz2019}.

\begin{theorem}\textup{\cite[Theorem 5.1]{GoodOprochaPuljiz2019}}
Let $X$ be a compact metric space and $f \colon X \to X$ a continuous onto map. Then $(X,f)$ has limit shadowing if and only if $(\varprojlim(X,f), \sigma)$ has limit shadowing.
\end{theorem}

\subsection{Tychonoff product}

\begin{theorem}
Let $\Lambda$ be an arbitrary indexation set and, for each $\lambda \in \Lambda$, let $(X_\lambda, f_\lambda)$ be a compact Hausdorff system with limit shadowing. Then the product system $(X,f)$, where $X=\prod_{\lambda \in \Lambda} X_\lambda$, has limit shadowing.
\end{theorem}

\begin{proof}
Let $(x_i)_{i \in \omega}$ be an asymptotic pseudo orbit in $X$. Then, for any $\lambda \in \Lambda$, $\left(\pi_\lambda(x_i)\right)_{i \in \omega}$ is an asymptotic pseudo orbit in $X_\lambda$. By limit shadowing in these component spaces, for each $\lambda$ there exists $y_\lambda \in X_\lambda$ which limit shadows $\left(\pi_\lambda(x_i)\right)_{i \in \omega}$. Let $y \in X$ be such that $\pi_\lambda(y)=y_\lambda$ for any $\lambda \in \Lambda$. We claim $y$ limit shadows $(x_i)_{i \in \omega}$. 

Let $E \in \mathscr{U}$ be given; this entourage is refined by one of the form
\[ \prod _{\lambda \in \Lambda} E_\lambda,\]
where $E_\lambda \in \mathscr{U}_\lambda$ for all $\lambda \in \Lambda$ and $E_\lambda =X_\lambda \times X_\lambda$ for all but finitely many of the $\lambda$'s. Let $\lambda_j$, for $1\leq j \leq k$, be precisely those elements in $\Lambda$ for which $E_\lambda  \neq X_\lambda \times X_\lambda$ (if there are no such elements then we are done). For each such $j$, let $M_j \in \mathbb{N}$ be such that for any $n \geq M_j$ $(f^n(y_{\lambda_j}) \pi_{\lambda_j}(x_n)) \in E_{\lambda_j}$. Take $M \coloneqq \max_{1 \leq j \leq k} M_j$. It follows that, for any $n \geq M$, $(f^n(y),x_n) \in E$.
\end{proof}


\section{Preservation of s-limit Shadowing}

The definition of limit shadowing was extended in \cite{LeeSakai} to a property the authors called s-limit shadowing. This was done to accommodate the fact that many systems exhibit limit shadowing but not shadowing \cite{Kulczycki,Pilyugin}.

Recall that a system $(X, f)$ is said to have \textit{s-limit shadowing} if for any $E\in \mathscr{U}$ there exists $D \in \mathscr{U}$ such that the following two conditions hold: 
\begin{enumerate}
\item every $D$-pseudo-orbit is $E$-shadowed, and
\item every asymptotic $D$-pseudo-orbit is asymptotically $E$-shadowed.
\end{enumerate}

Thus, part of what it means for a system to have s-limit shadowing is that it also has shadowing. It is a standard result in the theory of shadowing \cite{Pilyugin} that a compact metric dynamical system $(X,f)$ has shadowing if and only if for any $\epsilon>0$ there is a $\delta>0$ such that every finite $\delta$-pseudo orbit $(x_0,\ldots,x_n)$ is $\epsilon$-shadowed by some $x\in X$ (we call this property \textit{finite shadowing}). This extends to the compact Hausdorff setting: a compact Hausdorff dynamical system $(X,f)$ has shadowing if and only if for any $E \in \mathscr{U}$ there is a $D \in \mathscr{U}$ such that every finite $D$-pseudo orbit $(x_0,\ldots,x_n)$ is $E$-shadowed by some $x\in X$. This fact allows us to  make the observation (Theorem \ref{THMsLimObsolete}) that for a large class of systems, the definition of s-limit shadowing can be simplified.

\begin{theorem}\label{THMsLimObsolete}
Suppose $X$ is a compact Hausdorff space. $(X,f)$ has s-limit shadowing if and only if  for any $E\in \mathscr{U}$ there exists $D \in \mathscr{U}$ such that 
every asymptotic $D$-pseudo-orbit is asymptotically $E$-shadowed.

In particular, if $X$ is a compact metric space, then $(X,f)$ has s-limit shadowing if and only if for any $\epsilon>0$ there exists $\delta>0$  such that every asymptotic $\delta$-pseudo-orbit is asymptotically $\epsilon$-shadowed.
\end{theorem}

\begin{proof}
Condition (1) simply says that part of what it means for a system to have s-limit 
shadowing is that it has shadowing. Suppose that 
$(X,f)$ satisfies condition (2). Let $E \in \mathscr{U}$ be given and take a corresponding $D \in \mathscr{U}$. 
Let $(x_0,x_1,\ldots, x_m)$ be a finite $D$-pseudo orbit in $X$. Then
\[(x_0, x_1, \ldots, x_m, f(x_m), f^2(x_m), \ldots, f^k(x_m), \ldots),\]
is an asymptotic $D$-pseudo orbit. By condition (2) this is 
asymptotically $E$-shadowed by a point, say $x$. In particular 
$(f^i(x),x_i) \in \mathscr{U}$ for all $i \in \{0,1, \ldots, m\}$; hence $(X,f)$ has finite shadowing and thereby shadowing.
\end{proof}

Since our space is compact Hausdorff throughout this paper, it follows from Theorem \ref{THMsLimObsolete} that when checking for s-limit shadowing it suffices to verify whether or not condition (2) in the definition holds.

\subsection{Induced map on the hyperspace of compact sets}

\begin{theorem}\label{ThmSLimShadHyp} Let $X$ be a compact Hausdorff space and let $f \colon X \to X$ be a continuous function. If $(2^X,2^f)$ has s-limit shadowing then $(X,f)$ has s-limit shadowing.
\end{theorem}

\begin{proof}
Let $E \in \mathscr{U}$ be given. Let $D \in 2^\mathscr{U}$ correspond to $2^E$ in condition (2) of s-limit shadowing for $2^f$ and let $D_0 \in \mathscr{U}$ be such that $2^{D_0}\subseteq D$. Let $(x_i)_{i \in \omega}$ be an asymptotic $D_0$-pseudo-orbit in $X$. Then $(\{x_i\})_{i \in \omega}$ is an asymptotic $D$-pseudo-orbit in $2^X$; this is asymptotically $2^E$-shadowed by a set $A \in 2^X$ for some $D \in \mathscr{U}$. Pick $a \in A$. It is easy to verify that $a$ asymptotically $E$-shadows $(x_i)_{i \in \omega}$.
\end{proof}

\subsection{Symmetric products}

The proof of Theorem \ref{ThmSLimShadSymProd} is very similar to that of Theorem \ref{ThmSLimShadHyp} and is thereby omitted.

\begin{theorem}\label{ThmSLimShadSymProd}
Let $X$ be a compact Hausdorff space, and let $f \colon X \to X$ be a continuous onto function. For any $n \geq 2$, if the symmetric product system $(F_n(X), f_n)$ has s-limit shadowing then $(X,f)$ has s-limit shadowing.
\end{theorem}
\begin{proof}
Omitted.
\end{proof}

\begin{theorem}
Let $X$ be a compact Hausdorff space and let $f \colon X \to X$ be a continuous function. If $(X,f)$ has s-limit shadowing then $(F_2(X),f_2)$ has s-limit shadowing.
\end{theorem}

\begin{proof}
Let $E \in 2^\mathscr{U}$ be given. 
Let $E_0 \in \mathscr{U}$ be such that $2^{E_0} \subseteq E$. Let $D \in \mathscr{U}$ correspond to $E_0$ in s-limit shadowing for $f$. We claim $2^D$ satisfies condition (2) of s-limit shadowing for $E$. Suppose that $(A_i)_{i \in \omega}$ is an asymptotic $2^D$-pseudo-orbit in $F_2(X)$. Write $A_i=\{x_i, y_i\}$; it is possible that, for some $i$, $x_i=y_i$. We may relabel the $x$'s and $y$'s where necessary to give asymptotic $D$-pseudo-orbits $(x_i)_{i\in \omega}$ and $(y_i)_{i \in \omega}$ in $X$. 
By s-limit shadowing there exist $x,y \in X$ which asymptotically $E_0$-shadow $(x_i)_{i \in \omega}$ and $(y_i)_{i \in \omega}$ respectively. Write $A=\{x,y\} \in F_2(X)$. It is now straightforward to verify that $A$ asymptotically $E$-shadows $(A_i)_{i \in \omega}$.

\end{proof}

\begin{remark}
Example \ref{ExTentMapDoesNotPreserveLimitShadToSymProducts} shows that, in general, symmetric products do not preserve s-limit shadowing for $n\geq3$.
\end{remark}

\subsection{Factor maps}

\begin{definition} Suppose $X$ and $Y$ are compact Hausdorff spaces and $f\colon X \to X$, $g \colon Y \to Y$ are continuous. A surjective semiconjugacy
$\varphi\colon X\to Y$ is \emph{ALA$\epsilon$P} iff for every $V \in \mathscr{U}_Y$ and $D \in \mathscr{U}_X$ there is $W \in \mathscr{U}_Y$ such that for every asymptotic $W$-pseudo-orbit $(y_i)$ in $Y$ there is an asymptotic $D$-pseudo-orbit $(x_i)$ in $X$ such that $(\varphi(x_i))$ asymptotically $V$-shadows $(y_i)$.
\end{definition}
If $X$ and $Y$ are compact metric spaces, then $\varphi$ is ALA$\epsilon$P if and only if for every $\epsilon>0$ and $\eta>0$ there is $\delta>0$ such that for every asymptotic $\delta$-pseudo-orbit $(y_i)_{i \in \omega}$ in $Y$ there is an asymptotic $\eta$-pseudo-orbit $(x_i)_{i \in \omega}$ in $X$ such that $(\varphi(x_i))_{i \in \omega}$ asymptotically $\epsilon$-shadows $(y_i)_{i \in \omega}$.

The proof of the following is similar to the proofs of Theorems \ref{ThmALPeventualShad}, \ref{ThmALPOrbShad}, \ref{ThmALPStrongOrbShad} and \ref{ThmALPFirstWeakShad} and is therefore omitted

\begin{theorem}\label{ThmALPsLimShad}
Suppose that $\varphi \colon (X,f)\to (Y,g)$ is a factor map.
	\begin{enumerate}
		\item If $(X,f)$ exhibits s-limit shadowing and $\varphi$ is \emph{ALA$\epsilon$P}, then $(Y,g)$ exhibits s-limit shadowing.
		\item If $(Y,g)$ exhibits s-limit shadowing, then $\varphi$ is \emph{ALA$\epsilon$P}.
	\end{enumerate}
\end{theorem}

\subsection{Inverse limits}
Whilst it remains unclear whether general inverse limit systems preserve s-limit shadowing, we note the following result proved by the first author \textit{et al} in \cite{GoodOprochaPuljiz2019}.

\begin{theorem}\textup{\cite[Theorem 5.1]{GoodOprochaPuljiz2019}}
Let $X$ be a compact metric space and $f \colon X \to X$ a continuous onto map. Then $(X,f)$ has s-limit shadowing if and only if $(\varprojlim(X,f), \sigma)$ has s-limit shadowing.
\end{theorem}

\subsection{Tychonoff product}

\begin{theorem}
Let $\Lambda$ be an arbitrary indexation set and, for each $\lambda \in \Lambda$, let $(X_\lambda, f_\lambda)$ be a compact Hausdorff system with s-limit shadowing. Then the product system $(X,f)$, where $X=\prod_{\lambda \in \Lambda} X_\lambda$, has s-limit shadowing.
\end{theorem}

\begin{proof}
Let $E \in \mathscr{U}$ be given; this entourage is refined by one of the form
\[ \prod _{\lambda \in \Lambda} E_\lambda,\]
where $E_\lambda \in \mathscr{U}_\lambda$ for all $\lambda \in \Lambda$ and $E_\lambda =X_\lambda \times X_\lambda$ for all but finitely many of the $\lambda$'s. Let $\lambda_i$, for $1\leq i \leq k$, be precisely those elements in $\Lambda$ for which $E_\lambda  \neq X_\lambda \times X_\lambda$ (if there are no such elements then we are done).

By s-limit shadowing in each component space, there exist entourages $D_{\lambda_i} \in \mathscr{U}_{\lambda_i}$ such that every asymptotic $D_{\lambda_i}$-pseudo-orbit is asymptotically $E_{\lambda_i}$-shadowed. Note that, for every $\lambda \in \Lambda \setminus\{\lambda_i \mid 1 \leq i \leq k\}$ every asymptotic pseudo-orbit is asymptotically $E_{\lambda}$-shadowed. For $\lambda \in \Lambda \setminus\{\lambda_i \mid 1 \leq i \leq k\}$ take $D_\lambda=X \times X$.
Let
\[ D \coloneqq \prod _{\lambda \in \Lambda} D_\lambda.\]

Now let $(x_j)_{j \in \omega}$ be an asymptotic $D$-pseudo-orbit. Then $(\pi_{\lambda_i}(x_j))_{j \in \omega}$ is an asymptotic $D_{\lambda_i}$-pseudo-orbit in $X_{\lambda_i}$, which is asymptotically $E_{\lambda_i}$-shadowed by a point $z_i \in X_{\lambda_i}$. Furthermore $(\pi_{\lambda}(x_j))_{j \in \omega}$ is an asymptotic pseudo-orbit in $X_\lambda$ which is asymptotically shadowed by a point $z_\lambda$. Let $z \in X$ be such that $\pi_\lambda(z)=z_\lambda$ for all $\lambda \in \Lambda \setminus\{\lambda_i \mid 1 \leq i \leq k\}$ and $\pi_{\lambda_i}(z)=z_{\lambda_i}$ for each $i$. It is easy to see that $z$ asymptotically $E$-shadows  $(x_j)_{j \in \omega}$.

\end{proof}

\section{Preservation of Orbital Limit Shadowing}

Orbital limit shadowing was introduced by Pilyugin and others in \cite{Pilyugin2007} and studied with regard to various types of stability. Good and Meddaugh \cite{GoodMeddaugh2016} show that this property is equivalent to one they call \textit{asymptotic orbital shadowing} (see Theorem \ref{ThmEquivOrbLimICT} and Definition \ref{DefnAsymOrbShad}). Recall that a system $(X,f)$ has the \textit{orbital limit shadowing} property if given any asymptotic pseudo-orbit  $(x_i)_{i\geq 0}\subseteq X$, there exists a point $x\in X$ such that
	\begin{align*}
		\omega((x_i)_{i\geq 0})=\omega(x).
	\end{align*}
Where $\omega((x_i)_{i\geq 0})$ is the set of limit points of the pseudo-orbit.

The following theorem, proved in \cite{GoodMeddaugh2016}, gives an equivalence between two notions of shadowing that we have defined in Section \ref{ShadowingTypes}. It is because of this equivalence that asymptotic orbital shadowing 
is omitted from the table of results. (NB. The authors \cite{GoodMeddaugh2016} prove the theorem below in a compact metric setting. Their result generalises to the case when the underlying space is a compact Hausdorff.)

\begin{theorem}\textup{\cite[Theorem 22]{GoodMeddaugh2016}}\label{ThmEquivOrbLimICT} Let $(X,f)$ be a compact Hausdorff dynamical system. Then the following are equivalent:
\begin{enumerate}
\item $f$ has the asymptotic orbital shadowing property;
\item $f$ has the orbital limit shadowing property; and
\item $\omega_f = ICT(f)$.
\end{enumerate}
\end{theorem}

In the above theorem $\omega_f$ is the set of $\omega$-limit sets of $f$, whilst $ICT(f)$ is the set of \textit{internally chain transitive sets}: a set $A \subseteq X$ is internally chain transitive if for any $D \in \mathscr{U}$ and any $x,y \in A$ there exists a sequence of points in $A$, called a $D$-chain, $(x=x_0,x_1,x_2,\ldots, x_n=y)$ such that $(f(x_i), x_{i+1}) \in D$ for every $ 0\leq i \leq n-1$.


\subsection{Induced map on the hyperspace of compact sets}

\begin{theorem}\label{OrbLimHyp}
Let $X$ be a compact Hausdorff space, and let $f \colon X \to X$ be a continuous onto function. If $(2^X, 2^f)$ witnesses orbital limit shadowing then $(X,f)$ experiences orbital limit shadowing.
\end{theorem}

We will use the fact that orbital limit shadowing is equivalent to asymptotic orbital shadowing (Theorem \ref{ThmEquivOrbLimICT}). Recall the following definition:
The system $(X,f)$ has the \textit{asymptotic orbital shadowing} property if given any asymptotic pseudo-orbit $(x_i)_{i\geq 0}\subseteq X$, there exists a point $x\in X$ such that for any $E \in \mathscr{U}$ there exists $N \in \mathbb{N}$ such that
	\begin{align*}
		(\overline{\{x_{N+i}\}_{i\geq 0}},\overline{\{f^{N+i}(x)\}_{i\geq 0}}) \in 2^E.
	\end{align*}

\begin{proof}
Let $(x_i)_{i \in \omega}$ be an asymptotic pseudo-orbit in $X$. Notice that $(\{x_i\})_{i \in \omega}$ is an asymptotic pseudo-orbit in the hyperspace $2^X$. Thus, by asymptotic orbital shadowing, there exists $A \in 2^X$ such that for any $E \in 2^\mathscr{U}$, there exists $N \in \mathbb{N}$ such that
	\begin{align*}
		(\overline{\{\{x_{N+i}\}\}_{i\geq 0}},\overline{\{f^{N+i}(A)\}_{i\geq 0}}) \in 2^E.
	\end{align*}

Pick $z \in A$ and let $D \in \mathscr{U}$, so $2^D \in 2^\mathscr{U}$. Let $E \in \mathscr{U}$ be such that $4E \subseteq D$ and let $N \in \mathbb{N}$ be such that 
	\begin{align*}
		(\overline{\{\{x_{N+i}\}\}_{i\geq 0}},\overline{\{(2^f)^{N+i}(A)\}_{i\geq 0}}) \in 2^{2^E}.
	\end{align*}
	Equivalently
\begin{equation}\label{EquationOrbLim1}
    \overline{\{\{x_{N+i}\}\}_{i\in\omega}} \subseteq B_{2^E}\left(\overline{\{(2^f)^{N+i}(A)\}_{i\in\omega}}\right)
\end{equation}
and
\begin{equation}\label{EquationOrbLim2} \overline{\{(2^f)^{N+i}(A)\}_{i\in\omega}}\subseteq B_{2^E}\left(\overline{\{\{x_{N+i}\}\}_{i\in\omega}}\right).
\end{equation}

We claim
	\begin{align*}
		(\overline{\{x_{N+i}\}_{i\geq 0}},\overline{\{f^{N+i}(z)\}_{i\geq 0}}) \in 4E \subseteq 2^D.
	\end{align*}

Indeed, suppose not. 

\textbf{Case i).} There exists $a \in \overline{\{x_{N+i}\}_{i\in\omega}}$ such that for any $b \in \overline{\{f^{N+i}(z)\}_{i\in\omega}}$ we have $(a,b) \notin 4E$. It follows that there exists $k \geq N$ such that $(x_k, f^i(z))\notin 2E$ for all $i \geq N$. We have from Equation (\ref{EquationOrbLim1}) that there exists $l \geq N$ such that $\left(f^l(A), \{x_k\}\right) \in 2^{E}$; in particular, for any $y \in f^l(A)$, $(y, x_k)\in E$, a contradiction.

\textbf{Case ii).} There exists $b \in \overline{\{f^{N+i}(z)\}_{i\in\omega}}$ such that for any $a \in \overline{\{x_i\}_{i\in\omega}}$ we have $(b,a) \notin 4E$. It follows that there exists $k \geq N$ such that $(f^k(z),x_i)\notin 2E_0$ for all $i \geq N$. We have from Equation (\ref{EquationOrbLim2}) that there exists $l \geq N$ such that $\left(f^k(A), \{x_l\}\right) \in 2^{E}$; in particular, for any $y \in f^k(A)$, $(y, x_l)\in E$, a contradiction.

It follows that 
\[\left(\overline{\{x_{N+i}\}_{i\in\omega}}, \overline{\{f^{N+i}(z)\}_{i\in\omega}}\right) \in 2^{4E} \subseteq 2^D.\]

\end{proof}
 
The following example shows that the converse to Theorem \ref{OrbLimHyp} is false.

\begin{example}\label{OrbLimHypExample}
Let $X$ be the circle $\mathbb{R}/\mathbb{Z}$ and let $f\colon X \to X$ be given by $x \mapsto x + \alpha$, where $\alpha$ is some fixed irrational number. Since $(X,f)$ is minimal it clearly has orbital limit shadowing. (Indeed, this follows as a simple corollary to Theorem \ref{ThmEquivOrbLimICT} since $\omega_f=ICT$ for minimal systems.) 
Let $x_0$ and $y_0$ be two antipodal points and suppose $\delta \in \mathbb{Q}$ with $0<\delta<1$. Then construct an asymptotic pseudo-orbit in $2^X$ recursively by the following rule: Let $A_0=\{x_0,y_0\}$ and, for all $i\in \omega\setminus\{0\}$, let $A_i=\{x_i,y_i\}:=\{f(x_{i-1})+\frac{\delta}{2i}, f(y_{i-1}) +\frac{\delta}{3i}\}$. We claim that this is not orbital limit shadowed. 
Suppose $A \in 2^X$ orbital limit shadows $( A_i ) _{i \in \omega}$; i.e.
\[\omega(A) =\omega\left((A_i)_{i \geq 0}\right).\]
First note that $\omega\left((A_i)_{i \geq 0}\right)=\{\{a,b\} \mid a,b \in X\}$. 
If $A$ is infinite then there will be infinite sets in its $\omega$-limit set. Therefore $A$ must be finite; let $n$ be its cardinality. If $n \geq 3$ then there will be sets of size larger than $2$ in its $\omega$-limit set. It follows that we must have $n=2$. Write $A=\{x,y\}$ for distinct points $x,y \in X$. Since $2^f$ is a minimal isometry it follows that 
\[\omega(A)=\{\{a,b\} \mid d(a,b)=d(x,y)\}.\]
Pick distinct points $a,b \in X$ with $d(a,b) \neq d(x,y)$. Then $\{a,b\} \in \omega\left((A_i)_{i \geq 0}\right)$ but $\{a,b\} \notin \omega(A)$, a contradiction.
\end{example}

\subsection{Symmetric products}

The proof of Theorem \ref{OrbLimSymProd} is very similar to that of Theorem \ref{OrbLimHyp} and is thereby omitted.

\begin{theorem}\label{OrbLimSymProd}
Let $X$ be a compact Hausdorff space, and let $f \colon X \to X$ be a continuous onto function. For any $n \geq 2$, if the symmetric product system $(F_n(X), f_n)$ witnesses orbital shadowing then system $(X,f)$ experiences orbital shadowing.
\end{theorem}
\begin{proof}
Omitted.
\end{proof}

\begin{remark}
The converse of Theorem \ref{OrbLimSymProd} is false. It is clear that Example \ref{OrbLimHypExample} may be suitably adjusted to provide a counterexample. Indeed, with sufficient adjustments, one can see that, for any $n \geq 2$, $(X,f)$ witnessing orbital shadowing does not generally imply that $(F_n(X),f_n)$ has orbital shadowing.
\end{remark}

\subsection{Factor maps}

\begin{definition}
Let $(X,f)$ and $(Y,g)$ be dynamical systems where $X$ and $Y$ are compact Hausdorff spaces. A surjective semiconjugacy $\varphi \colon X\to Y$ is \emph{oALAP} if for every asymptotic pseudo-orbit $(y_i)_{i \in \omega}\subseteq Y$, there exists an asymptotic-pseudo-orbit $(x_i)_{i \in \omega}\subseteq X$ such that for any $E \in \mathscr{U}_Y$ there exists $N \in \mathbb{N}$ such that for all $i \geq N$
	\begin{align*}
    	(\varphi(\overline{\{x_i\}_{i\in\omega}}),\overline{\{y_i\}_{i\in\omega}}) \in 2^E.
     \end{align*}
\end{definition}
If $X$ and $Y$ are compact metric,  then $\varphi$ is oALAP if and only if for every asymptotic pseudo-orbit $(y_i)_{i \in \omega}$ in $Y$ there exists an asymptotic pseudo-orbit $(x_i)_{i \in \omega}$ in $X$ such that for every $\epsilon>0$ there is  $N>0$ for which the Hausdorff distance 
$d_H(\varphi(\overline{\{x_i\}_{i\in\omega}}),\overline{\{y_i\}_{i\in\omega}})<\epsilon$ for any $i\geq N$.

Again the proof of the following theorem is  
similar to that of theorems  \ref{ThmALPeventualShad}, \ref{ThmALPOrbShad}, \ref{ThmALPStrongOrbShad} and \ref{ThmALPFirstWeakShad} bearing in mind the equivalence between orbital limit shadowing and asymptotic orbital shadowing (\cite[Theorem 22]{GoodMeddaugh2016}).

\begin{theorem}\label{ThmALPasympOrbLimShad}
Suppose that $\varphi \colon (X,f)\to (Y,g)$ is a factor map.
	\begin{enumerate}
		\item If $(X,f)$ exhibits orbital limit shadowing and $\varphi$ is oALAP, then $(Y,g)$ exhibits orbital limit shadowing.
		\item If $(Y,g)$ exhibits orbital limit shadowing, then $\varphi$ is oALAP.
	\end{enumerate}
\end{theorem}


\subsection{Tychonoff product}

A product of systems with orbital limit shadowing does not necessarily have orbital limit shadowing. The following example demonstrates this.

\begin{example}\label{ExampleProdOrbLimNotOrbLim}
For $i \in \{1,2\}$ let $X_i= \mathbb{R}/\mathbb{Z}$, $d_i$ be the shortest arc length metric on $X_i$ and $f_i \colon X_i \to X_i \colon x \mapsto x + \alpha \mod 1$, where $\alpha$ is some fixed irrational number. Equip the product space $X=X_1 \times X_2$ with the metric $d$ given by $d((a,b), (c,d))= \sup\{d_1(a, c), d_2(b,d)\}$. 
Now consider the product system $(X,f)$. Let $x_0$ and $y_0$ be two antipodal points and suppose $\delta \in \mathbb{Q}$ with $0<\delta<1$. Then construct an asymptotic pseudo-orbit in $X$ recursively by the following rule: Let $z_0=(x_0,y_0)$ and, for all $i\in \omega\setminus\{0\}$, let $z_i=(f(x_{i-1})+\frac{\delta}{2i}, f(y_{i-1}) +\frac{\delta}{3i})$ where $x_i=f(x_{i-1})+\frac{\delta}{2i}$ and $y_i=f(y_{i-1}) +\frac{\delta}{3i}$. We claim that this is not orbital limit shadowed. 
Suppose $z=(x,y) \in X$ orbital limit shadows $( z_i ) _{i \in \omega}$; i.e.
\[\omega(z) =\omega\left((z_i)\right).\]
It is easy to see that $\omega\left((z_i)\right))=\{(a,b) \mid a,b \in X\}$. 
\[\omega(z)\subseteq \{(a,b) \mid \min \{\lvert a-b\rvert, \lvert b-a \rvert\} =  \min \{\lvert x-y\rvert, \lvert y-x \rvert\}\},\]
where equality holds only when  $\min \{\lvert x-y\rvert, \lvert y-x \rvert\}\in \{0, \frac{1}{2}\}$. Therefore by picking $(a,b) \in X$ with $\min \{\lvert a-b\rvert, \lvert b-a \rvert\} \neq \min \{\lvert x-y\rvert, \lvert y-x \rvert\}$ then we get  $(a,b) \in \omega\left((z_i)_{i \geq 0}\right)$ but $(a,b) \notin \omega(z)$, a contradiction.
\end{example}

\section{Preservation of Inverse Shadowing}
The presence (or absence) of shadowing in a dynamical system tells us whether or not any given computed orbit is followed (to within some constant error) by a true trajectory. In a related fashion, one may wonder under what circumstances can actual trajectories be recovered, within a given accuracy, from pseudo-orbits. As observed elsewhere (e.g. \cite{Lee}), this relates to inverse shadowing. In this paper we limit our discussion to $\mathcal{T}_0$-inverse shadowing, however we note that weaker formulations of inverse shadowing can be given by restricting one's attention to certain classes of pseudo orbits (see for example \cite{Lee,LeePark}).

Recall that the system $(X,f)$ experiences \textit{inverse shadowing} if, for any $E \in \mathscr{U}$ there exists $D \in \mathscr{U}$ such that for any $x \in X$ and any $\varphi \in \mathcal{T}_0(f,D)$ there exists $y \in X$ such that $\varphi(y)$ $E$-shadows $x$; i.e.
\[ \forall k \in \omega, (\varphi(y)_k, f^k(x))\in E.\]



Compact metric versions of the two results below may be found in \cite{GoodMitchellThomas}; the authors remark that the compact metric versions extend to these two results. 

\begin{lemma}\textup{\cite[Theorem 2.1]{GoodMitchellThomas}}\label{lemmaUniformInverseShadReform}
A continuous function $f \colon X \to X$ has inverse shadowing if and only if for any  $E \in \mathscr{U}$ there exists $D \in \mathscr{U}$ such that for any $x \in X$ there exists $y \in X$ such that for any $\varphi \in \mathcal{T}_0(f,D)$, $\varphi(y)$ $E$-shadows $x$; i.e.
\[ \forall k \in \omega, (\varphi(y)_k, f^k(x))\in E.\]
\end{lemma}

\begin{lemma}\textup{\cite[Corollary 2.2]{GoodMitchellThomas}}\label{lemmaUniformInverseShadBallReform}
A continuous function $f \colon X \to X$ has inverse shadowing if and only if for any $E \in \mathscr{U}$ there exists $D \in \mathscr{U}$ such that for any $x \in X$ there exists $y \in X$ such that for any $y^\prime \in B_D(y)$ and any $\varphi \in \mathcal{T}_0(f,D)$, $\varphi(y^\prime)$ $E$-shadows $x$; i.e.
\[ \forall k \in \omega, \, (\varphi(y^\prime)_k, f^k(x))\in E.\]
\end{lemma}

 
Recall the open cover formulation of inverse shadowing from Section \ref{ShadowingTypes}, which coincides with the uniform definition in the presence of compactness.
The system $(X,f)$ experiences \textit{inverse shadowing} if, for any finite open cover $\mathcal{U}$ there exists a finite open cover $\mathcal{V}$ such that for any $x \in X$ and any $\varphi \in \mathcal{T}_0(f,\mathcal{U})$ there exists $y \in X$ such that $\varphi(y)$ $\mathcal{U}$-shadows $x$; i.e.
\[ \forall k \in \omega \, \exists U \in \mathcal U : \varphi(y)_k, f^k(x)\in U.\]

The following lemma is an open cover version of Lemma \ref{lemmaUniformInverseShadReform}.
\begin{lemma}\label{lemmaOpenCoverInverseShadReform}
A continuous function $f \colon X \to X$ has inverse shadowing if and only if for any finite open cover $\mathcal{U}$ there exists a finite open cover $\mathcal{V}$ such that for any $x \in X$ there exists $y \in X$ for any $\varphi \in \mathcal{T}_0(f,\mathcal{V})$ $\varphi(y)$ $\mathcal{U}$-shadows $x$; i.e.
\[ \forall k \in \omega \, \exists U \in \mathcal{U} : \varphi(y^\prime)_k, f^k(x)\in U.\]
\end{lemma}

\subsection{Induced map on the hyperspace of compact sets}

\begin{theorem} Let $X$ be a compact Hausdorff space and let $f \colon X \to X$ be a continuous function. Then $(X,f)$ has $\mathcal{T}_0$-inverse shadowing if and only if $(2^X,2^f)$ has $\mathcal{T}_0$-inverse shadowing.
\end{theorem}

\begin{proof}
Suppose that $f$ has inverse shadowing. Let $2^E \in \mathscr{B}_\mathscr{U}$ (so $E \in \mathscr{U}$). Let $E_0 \in \mathscr{U}$ be such that $2E_0 \subseteq E$ and take $D \in \mathscr{U}$ be as in Lemma \ref{lemmaUniformInverseShadBallReform}. That is, for any $x \in X$ there exists $y \in X$ such that for any $y^\prime \in B_D(y)$ and any $\varphi \in \mathcal{T}_0(f,D)$ we have $\varphi(y^\prime)$ $E_0$-shadows $x$. 

Choose $A \in 2^X$ and let $\psi \in \mathcal{T}_0(2^f, 2^D)$. For each $x \in A$ let $y_x \in X$ be such that for any $y^\prime \in B_D(y)$ and any $\varphi \in \mathcal{T}_0(f,D)$ we have $\varphi(y^\prime)$ $E_0$-shadows $x$. Define 
\[C \coloneqq \overline{\bigcup_{x \in A} \{y_x\}}.\]
It is easy to verify that $\psi(C)$ $2^E$-shadows $A$. Indeed, let $k \in \omega$ be given. 
Choose $a \in (2^f)^k(A)$ and $a^\prime \in A$ such that $f^k(a^\prime)=a$. Since $\psi$ is a $2^D$-method for $2^f$, for each $i \in \{1,\ldots, k\}$ there exists $c_i \in \psi(C)_k$ such that $(f(c_i),c_{i+1}) \in D$ and $(f(y_{a^\prime}), c_1) \in D$. Denote by $c_0=y_{a^\prime}$ so that by definition of $y_{a^\prime}$, $(c_i, f^i(a^\prime)) \in E_0$ for all $i \in \{0,\ldots, k\}$. In particular it follows that $(a,c_k) \in E$. As $a$ was picked arbitrarily from $(2^f)^k(A)$ it follows that $(2^f)^k(A) \subseteq B_E(\psi(C)_k)$. 
Now choose $c \in \psi(C)_k$. By construction, there exists a sequence $(c_0,c_1, \ldots, c_k)$ with $c_i \in \psi(C)_i$ for each $0 \leq i \leq k$ and $c_k=c$ for which $(f(c_i), c_{i+1}) \in D$. Furthermore, there exist $x \in A$ for which $c_0 \in B_D(y_x)$. It follows that $(c_i, f^i(x)) \in E_0$ for each $0 \leq i \leq k$. In particular it follows that $(c, f^k(x)) \in E$. As $c$ was picked arbitrarily from $\psi(C)_k$ it follows that $\psi(C)_k \subseteq B_E((2^f)^k(A))$. 

Now suppose that $2^f$ has inverse shadowing. Let $E \in \mathscr{U}$ be given and let $E_0 \in \mathscr{U}$ be such that $2E_0 \subseteq E$. Let $D_0, D \in \mathscr{U}$ be such that $2^D$ satisfies the inverse shadowing condition for $2^{E_0}$ and $2D_0 \subseteq D$. Now pick $\varphi \in \mathcal{T}_0(f,D_0)$. Construct a method $2^\varphi$ as follows:
For any $k \in \omega$ and any $A \in 2^X$
\[2^\varphi(A)_k\coloneqq \overline{\bigcup_{x \in A}\{\varphi(x)_k\}}.\]
Clearly $2^\varphi \in \mathcal{T}_0(f, 2^D)$. Let $x \in X$ be given. Then $\{x\} \in 2^X$. By inverse shadowing there exists $A \in 2^X$ such that $2^\varphi(A)$ $2^{E_0}$-shadows $\{x\}$. Pick $a \in A$. It is easy to verify that $\varphi(a)$ $E$-shadows $x$.

\end{proof}

\subsection{Symmetric products}

\begin{theorem} Let $X$ be a compact Hausdorff space and let $f \colon X \to X$ be a continuous function. Then $(X,f)$ has $\mathcal{T}_0$-inverse shadowing if and only if $(F_n(X),f_n)$ has $\mathcal{T}_0$-inverse shadowing for all $n \geq 2$.
\end{theorem}

\begin{proof}
Suppose that $(X,f)$ has inverse shadowing and fix $n \geq 2$. Let $2^E \in \mathscr{B}_\mathscr{U}$ (so $E \in \mathscr{U}$). Let $E_0 \in \mathscr{U}$ be such that $2E_0 \subseteq E$ and take $D \in \mathscr{U}$ be as in Lemma \ref{lemmaUniformInverseShadReform}; that is for any $x \in X$ there exists $y \in X$ such that for any $\varphi \in \mathcal{T}_0(f,D)$ we have $\varphi(y)$ $E_0$-shadows $x$. 

Pick $A \in F_n(X)$ and let $\psi \in \mathcal{T}_0(f_n, 2^D \cap F_n(X))$. For each $x \in A$ let $y_x \in X$ be such that for any $\varphi \in \mathcal{T}_0(f,D)$, $\varphi(y)$ $E_0$-shadows $x$. Define
\[C \coloneqq\bigcup_{x \in A} \{y_x\}.\]
Note the $\lvert C \rvert \leq \lvert A \rvert$ so $C \in F_n(X)$.
It is easy to verify that $\psi(C)$ $2^E$-shadows $A$. Indeed, let $k \in \omega$ be given. 
Pick $a \in f_n^k(A)$ and $a^\prime \in A$ such that $f^k(a^\prime)=a$. Since $\psi$ is a $\left(2^D \cap (F_n(X) \times F_n(X))\right)$-method for $f_n$, for each $i \in \{1,\ldots, k\}$ there exists $c_i \in \psi(C)_k$ such that $(f(c_i),c_{i+1}) \in D$ and $(y_{a^\prime}, c_1) \in D$. Denote $c_0=y_{a^\prime}$ so that by definition of $y_{a^\prime}$, $(c_i, f^i(a^\prime)) \in E_0$ for all $i \in \{0,\ldots, k\}$. In particular, it follows that $(a,c_k) \in E$. As $a$ was picked arbitrarily from $f_n ^k(A)$ it follows that $f^k(A) \subseteq B_E(\psi(C)_k)$. 
Now pick $c \in \psi(C)_k$. By construction, there exists a sequence $(c_0,c_1, \ldots, c_k)$ with $c_i \in \psi(C)_i$ for each $0 \leq i \leq k$ and $c_k=c$ for which $(f(c_i), c_{i+1}) \in D$. Furthermore, there exist $x \in A$ for which $c_0 =y_x$. It follows that $(c_i, f^i(x)) \in E_0$ for each $0 \leq i \leq k$. In particular it follows that $(c, f^k(x)) \in E$. As $c$ was picked arbitrarily from $\psi(C)_k$ it follows that $\psi(C)_k \subseteq B_E(f_n ^k(A))$. 

Now suppose that $f_n$ has inverse shadowing. Let $E \in \mathscr{U}$ be given and let $E_0 \in \mathscr{U}$ be such that $2E_0 \subseteq E$. Let $D \in \mathscr{U}$ be such that $2^D$ satisfies the inverse shadowing condition for $2^{E_0}$. Now pick $\varphi \in \mathcal{T}_0(f,D)$. 
Construct method $2^\varphi$ as follows:
For any $k \in \omega$ and any $A \in F_n(X)$
\[2^\varphi(A)_k\coloneqq \bigcup_{x \in A}\{\varphi(x)_k\}.\]
Clearly $2^\varphi \in \mathcal{T}_0\left(f_n, (2^D \cap (F_n(X)\times F_n(X)))\right)$. Let $x \in X$ be given. Then $\{x\} \in F_n(X)$. By inverse shadowing there exists $A \in F_n(X)$ such that $2^\varphi(A)$ $2^{E_0}$-shadows $\{x\}$. Pick $a \in A$. It is easy to verify that $\varphi(a)$ $E$-shadows $x$.
\end{proof}

\subsection{Inverse limits}
The following theorem generalises part (2) of Theorem 7 in \cite{Barzanouni}, where the author shows that the induced shift space on a system with inverse shadowing also has inverse shadowing. 
\begin{theorem} 
Let $(X,f)$ be conjugate to a surjective inverse limit system comprised of maps with  $\mathcal{T}_0$-inverse shadowing on compact Hausdorff spaces. Then $(X,f)$ has $\mathcal{T}_0$-inverse shadowing. 
\end{theorem}

\begin{proof}
We use the reformulation of inverse shadowing given in Lemma \ref{lemmaOpenCoverInverseShadReform}.

Let $(\Lambda, \geq)$ be a directed set. For each $\lambda \in \Lambda$, let $(X_\lambda, f_\lambda)$ be a dynamical system on a compact Hausdorff space with inverse shadowing and let $((X_\lambda, f_\lambda), g^\eta _\lambda)$ be an surjective inverse system. Without loss of generality $(X,f)=(\varprojlim \{ X_\lambda, g_\lambda ^\eta\}, f )$.

Let $\mathcal{U}$ be a finite open cover of $X$. Since $X =\mathcal{W}=\varprojlim \{ X_\lambda, g_\lambda ^\eta\}$ there exist $\eta \in \Lambda$ and a finite open cover $\mathcal{W}_\eta$ of $X_\eta$ such that $\mathcal{W}\coloneqq \{\pi_\eta ^{-1}(W) \cap X \mid W \in \mathcal{W}_\eta\}$ refines $\mathcal{U}$. By inverse shadowing for $(X_\eta, f_\eta)$, and by Lemma \ref{lemmaOpenCoverInverseShadReform}, there exists a finite open cover $\mathcal{V}_\eta$ of $X_\eta$ such that for any $x \in X_\eta$ there exists $y \in X_\eta$ such that for any $\varphi \in \mathcal{T}_0(f_\eta, \mathcal{V}_\eta)$ $\varphi(y)$ $\mathcal{W}_\eta$-shadows $x$. 

Take $\mathcal{V}=\{ \pi_\eta ^{-1}(V) \cap X \mid V \in \mathcal{V}_\lambda \}$. Pick $x \in X$ and take $\phi \in \mathcal{T}_0(f,\mathcal{V})$. Let $y \in X_\eta$ be such that for any $\varphi \in \mathcal{T}_0(f_\eta, \mathcal{V}_\eta)$ $\varphi(y)$ $\mathcal{W}_\eta$-shadows $\pi_\eta(x)$. Notice that, if $z \in \pi_\eta ^{-1}(y) \cap X$ then $\left(\pi_\eta(\phi(z)_i)\right)_{i \in \omega}$ is a $\mathcal{V}_\eta$-pseudo-orbit starting from $y$; hence it $\mathcal{W}_\eta$-shadows $\pi_\eta(x)$. Therefore, taking any such $z$, we have $\phi(z)$ $\mathcal{W}$-shadows $x$. Since $\mathcal{W}$ is a refinement of $\mathcal{U}$ the result follows.
\end{proof}

\subsection{Tychonoff product}

\begin{theorem}
Let $\Lambda$ be an arbitrary index set and, for each $\lambda \in \Lambda$, let $(X_\lambda, f_\lambda)$ be a compact Hausdorff system with $\mathcal{T}_0$-inverse shadowing. Then the product system $(X,f)$, where $X=\prod_{\lambda \in \Lambda} X_\lambda$, has $\mathcal{T}_0$-inverse shadowing.
\end{theorem}

\begin{proof}
Let $E \in \mathscr{U}$ be given; this entourage is refined by one of the form
\[ \prod _{\lambda \in \Lambda} E_\lambda,\]
where $E_\lambda \in \mathscr{U}_\lambda$ for all $\lambda \in \Lambda$ and $E_\lambda =X_\lambda \times X_\lambda$ for all but finitely many of the $\lambda$'s. Let $\lambda_j$, for $1\leq j \leq k$, be precisely those elements in $\Lambda$ for which $E_\lambda  \neq X_\lambda \times X_\lambda$ (if there are no such elements then we are done).  By inverse shadowing in each component space, there exist entourages $D_{\lambda_i} \in \mathscr{U}_{\lambda_i}$ such that corresponding to the entourages $E_{\lambda_i}$ as in Lemma \ref{lemmaUniformInverseShadReform}. 

Let 
\[D \coloneqq \prod_{\lambda \in \Lambda} D_\lambda \]
where
\[D_\lambda = \left\{\begin{array}{lll}
X\times X  & \text{ if } & \forall i \, \lambda \neq \lambda_i
\\D_{\lambda_i} & \text{ if } & \exists i : \lambda=\lambda_i\,.
\end{array}\right.\]

Now pick $x \in X$ and pick $\phi \in \mathcal{T}_0(f,D)$. For each $1 \leq i \leq k$ let $y_i \in X_{\lambda_i}$ be as in Lemma \ref{lemmaUniformInverseShadReform} for $E_{\lambda_i}$, $D_{\lambda_i}$ and $\pi_{\lambda_i}(x)$. Pick a point $y \in X$ such that $\pi_{\lambda_i}(y)= y_i$ for each $1\leq i \leq k$. It can be seen that $\phi(y)$ $E$-shadows $x$.

\end{proof}

\begin{acknowledgements} The first author gratefully acknowlegdes support from the European Union through funding the H2020-MSCA-IF-2014 project ShadOmIC (SEP-210195797) and from the Institut Mittag-Leffler during the ‘Thermodynamic Formalism - Applications to Geometry, Number Theory, and Stochastics’ workshop.
\end{acknowledgements}

\bibliographystyle{plain} 
\bibliography{bib}

\end{document}